\documentclass{amsart}

\usepackage[all]{xy}
\usepackage{amsmath}
\usepackage{hyperref}
\usepackage{amsfonts,graphics,amsthm,amsfonts,amscd,latexsym}
\usepackage{epsfig}
\usepackage{flafter}
\usepackage{mathtools}
\usepackage{comment}
\usepackage{stmaryrd}
\usepackage{tabularx}
\usepackage{pst-node}
\usepackage{auto-pst-pdf}
\usepackage{tikz-cd}
\hypersetup{
    colorlinks=true,    
    linkcolor=blue,          
    citecolor=blue,      
    filecolor=blue,      
    urlcolor=blue           
}

\usepackage{pgfplots}
\pgfplotsset{compat=1.15}
\usepackage{mathrsfs}
\usepackage{tikz}
\usetikzlibrary{graphs,positioning,arrows,shapes.misc,decorations.pathmorphing}

\tikzset{
    >=stealth,
    every picture/.style={thick},
    graphs/every graph/.style={empty nodes},
}

\tikzstyle{vertex}=[
    draw,
    circle,
    fill=black,
    inner sep=1pt,
    minimum width=5pt,
]
\usepackage[position=top]{subfig}
\usepackage{amssymb}
\usepackage{color}

\calclayout

\usetikzlibrary{decorations.pathmorphing}
\tikzstyle{printersafe}=[decoration={snake,amplitude=0pt}]

\newcommand{\supp}{\operatorname{supp}}

\newcommand{\vol}{\operatorname{vol}}

\newcommand{\pp}{\mathbb{P}}

\newcommand{\qq}{\mathbb{Q}}
\newcommand{\zz}{\mathbb{Z}}

\newcommand{\kk}{\mathbb{K}}

\DeclareMathOperator{\coeff}{coeff}

\usepackage{tikz}
\usetikzlibrary{matrix,arrows,decorations.pathmorphing}
\usetikzlibrary{arrows}

\definecolor{uuuuuu}{rgb}{0.26666666666666666,0.26666666666666666,0.26666666666666666}

  \newtheorem{theorem}{Theorem}[section]
  \newtheorem{lemma}[theorem]{Lemma}
  \newtheorem{proposition}[theorem]{Proposition}

  \newtheorem{definition}[theorem]{Definition}
  \newtheorem{example}[theorem]{Example}

  \newtheorem{question}[theorem]{Question}

\theoremstyle{remark}

\numberwithin{equation}{section}



\begin{document}

\title[Calabi--Yau pairs of complexity two]{Calabi--Yau pairs of complexity two}

\author[J.~Moraga]{Joaqu\'in Moraga}
\address{UCLA Mathematics Department, Box 951555, Los Angeles, CA 90095-1555, USA
}
\email{jmoraga@math.ucla.edu}

\author[J.I.~Y\'a\~nez]{Jose Ignacio Y\'a\~nez}
\address{UCLA Mathematics Department, Box 951555, Los Angeles, CA 90095-1555, USA
}
\email{yanez@math.ucla.edu}

\thanks{The second author was partially supported by NSF research grant DMS-2443425.}

\begin{abstract}
A Calabi--Yau pair of index one and complexity zero is toric. Furthermore, a Calabi--Yau pair of index one and complexity one is of cluster type.
In this article, we study Calabi--Yau pairs of index one and complexity two.
We develop machinery to decide whether a Calabi--Yau of complexity two is of cluster type.
This approach reduces the problem to studying del Pezzo fibrations over toric varieties. 
We apply this to the setting of Gorenstein del Pezzo surfaces of Picard rank one. We prove that such a surface $X$ is cluster type if and only if $X$ has only $A$-type singularities
and either $\vol(X)>1$ or $|X^{\rm sing}|\leq 3$.
\end{abstract}
\maketitle

\setcounter{tocdepth}{1}
\tableofcontents

\section{Introduction}

Calabi--Yau varieties are one of the three building blocks of algebraic varieties.
Calabi--Yau pairs, as their name indicates, are a logarithmic version of Calabi--Yau varieties. 
More precisely, a Calabi--Yau pair is a couple $(X,B)$ with good singularities (log canonical) and trivial curvature, i.e., $K_X+B\equiv 0$.
Calabi--Yau pairs appear naturally when deforming Calabi--Yau varieties~\cite{Gro05}.
Moreover, every Fano variety $X$ admits several Calabi--Yau pair $(X,B)$ structures.
Toric Calabi--Yau pairs\footnote{This means that $X$ is a toric variety and $B$ is the reduced sum of all the torus invariant divisors.}
are the ones that we can understand the best due to their combinatorial nature.
For a toric Calabi--Yau pair $(X,B)$, the number 
\begin{equation}\label{eq:comp}
\dim(X)+\dim_\qq{\rm Cl}_\qq(X) -|B|
\end{equation}
is equal to zero, where $|B|$ stands for the sum of the coefficients of $B$ (see, e.g.,~\cite{CLS11}).
For a Calabi--Yau pair $(X,B)$ the value in~\eqref{eq:comp} is known as the {\em complexity}, and it is denoted by $c(X,B)$.
In~\cite{BMSZ18}, the authors show that the complexity of a Calabi--Yau pair is non-negative, and whenever $c(X,B)<1$ the underlying variety $X$ is toric, although $B$ may not be torus invariant.
In~\cite{EF24}, Enwright and Figueroa show that $c(X,B)=0$ if and only if $X$ is toric and $B$ is a weighted sum of different toric boundaries $B_1,\dots,B_k$ corresponding to different maximal tori of ${\rm Aut}(X)$. The previous theorem gives us a good understanding of Calabi--Yau pairs of complexity zero and arbitrary index.
In recent years, there has been some interest in understanding how the complexity reflects the geometry of the Calabi--Yau pair, beyond the case of complexity zero. Mauri and the first author have found connections between the complexity and dual complexes~\cite{MM24}.
In the works~\cite{Mor24a,Mor24b}, the first author established a link between the complexity and existence of non-trivial birational Mori fiber spaces.

In~\cite{ELY25}, Enwright, Li, and the second author prove that a Calabi--Yau pair of complexity one $(X,B)$ satisfies that $X$ is a toric hypersurface.
Furthermore, they prove that whenever such a pair $(X,B)$ has index one it is of cluster type. Cluster type pairs are closely related to toric pairs (see Definition~\ref{def:cluster-type}).
Indeed, when $X$ is a $\qq$-factorial Fano variety and $(X,B)$ is cluster type, the open Calabi--Yau variety $U=X\setminus B$ is covered by algebraic tori up to a codimension two subset (see~\cite[Theorem 1.3.(4)]{EFM24}). The main theorem of~\cite{ELY25} is the following.

\begin{theorem}\label{thm:comp1-ct}
Let $(X,B)$ be a Calabi--Yau pair of index one and complexity one. Then $(X,B)$ is of cluster type. 
\end{theorem}

Cluster type varieties play an important role among rational varieties.
On one hand, these varieties share combinatorial similarities to toric varieties. 
On the other hand, there are interesting families of
cluster type varieties.
Deciding whether a variety is of cluster type tends to be a challenging problem.
The main aim of this article is to introduce methods to decide whether a Calabi--Yau pair of complexity two is of cluster type. This method leads to new results even in classic cases. For instance, it allows us to decide when some singular surfaces are of cluster type (see Theorem~\ref{thm:dP-Gor}).

\subsection{Projective surfaces}

For a cluster type surface $X$
there is an anti-canonical cycle $B$
of rational curves.
On the other hand, in the case of smooth projective surfaces,
the existence of such anti-canonical cycle of rational curves
is equivalent to the cluster type condition:

\begin{theorem}\label{thm:ct-surf}
Let $X$ be a smooth projective surface. 
Then,  $X$ is of cluster type if and only if
we can find a cycle of rational curves $B\in |-K_X|$. 
\end{theorem}

The previous theorem was proved by Gross, Hacking, and Keel 
when studying Looijenga pairs 
in the context of Mirror Symmetry for Calabi--Yau pairs (see \cite[Proposition 1.3]{GHK15}).
As an application, one deduces that all del Pezzo surfaces of degree
at least two
and general del Pezzo surfaces of degree one are cluster type (see \cite[Theorem 2.1 and Remark 2.2]{ALP23}).  
The situation is more complicated when we consider singular surfaces.

After the case of smooth surfaces, the next simpler case
is such of Gorenstein surfaces. 
In~\cite{EFM24}, the authors prove that a cluster type surface has only $A$-type singularities.
However, Gorenstein projective surfaces with the same 
$A$-type singularities may or may not be cluster type (see Example~\ref{A-ct-nct}).
In~\cite{MZ88}, Miyanishi and Zhang classified Gorenstein del Pezzo 
surfaces of Picard rank one in 31 families. 
These surfaces play a central role in the classification 
of Gorenstein surfaces as they often appear
as the terminal model of a Minimal Model Program.
As an application of our methods to this setting, we prove the following theorem.

\begin{theorem}\label{thm:dP-Gor}
Let $X$ be a Gorenstein del Pezzo surface of Picard rank one. 
Then $X$ is of cluster type
if and only if $X$ only has $A$-type singularities
and at least one of the two following conditions hold:
\begin{itemize}
\item we have $\vol(X)>1$, or
\item we have $|X^{\rm sing}|<4$. 
\end{itemize}
\end{theorem}

As a consequence, there are $14$ isomorphisms classes of 
Gorenstein del Pezzo surfaces of Picard rank one
which are cluster type. Among Gorenstein del Pezzo surfaces of rank one and $A$-type singularities, there are precisely two with volume one and four singular points, namely 
$X(4A_2)$ and $X(2A_1+2A_3)$. These are the only Gorenstein del Pezzo surfaces of rank one and $A$-type singularities which are not cluster type.
We expect that our method also apply to more singular Fano surfaces, however, 
a more detailed analysis of the singularities of $X$ is necessary. In Theorem~\ref{thm:complete-Gor-dP}, 
we prove a more complete statement that allows us to detect whether 
any pair $(X,B)$, with $X$ as in Theorem~\ref{thm:dP-Gor}, is of cluster type or not.
Theorem~\ref{thm:complete-Gor-dP} is more intricate and depends 
on several factors; the number of components of $B$, the volume of $X$,
and the singularities of $X$ contained in $B$. 

The main tool to prove Theorem~\ref{thm:dP-Gor} and Theorem~\ref{thm:complete-Gor-dP} 
is Theorem~\ref{thm:ct-standard-model-Pic-1} below. Theorem~\ref{thm:ct-standard-model-Pic-1} allow us to understand the cluster type condition, for pairs of complexity two, in terms of volumes and resolutions of singularities.

\subsection{Higher-dimensional Calabi--Yau pairs}
Beyond the case of complexity one, Alves da Silva, Figueroa, and the first author studied quartic surfaces $B\subset \pp^3$ for which the pair $(\pp^3,B)$ is of cluster type, i.e., $\pp^3\setminus B$ contains an algebraic torus. In~\cite{AdSFM24}, the authors classify which pairs with $c(\pp^3,B)=2$ are of cluster type.
This motivates the study of Calabi--Yau pairs of index one and complexity two in a broader context.

In this article, we study Calabi--Yau pairs $(X,B)$ of index one and complexity two.
In Example~\ref{ex:comp2-not-rat}, we show that these pairs may not be cluster type, indeed, the underlying variety $X$ may not even be rational. Our main aim is to understand whether such pair $(X,B)$ is of cluster type or not. 
Our first result is a structural theorem
for Calabi--Yau pairs of index one and complexity two.

\begin{theorem}\label{thm:comp2-structure}
Let $X$ be a Fano type variety and let $(X,B)$ be a Calabi--Yau pair
of index one and complexity two.
There exists a crepant birational map $\phi\colon (X',B')\dashrightarrow (X,B)$
satisfying one of the following.
\begin{enumerate}
    \item[(i)] There is a finite crepant morphism $(Y,B_Y)\rightarrow (X',B')$ of degree at most $2$ for which $(Y,B_Y)$ is of cluster type; or
    \item[(ii)] there are two crepant fibrations
\[
\xymatrix{ 
(X',B') \ar[r]^-{f_1} & (X'',B'') 
\ar[r]^-{f_2} & (T,B_T) 
}
\]
where $f_1$ is a composition of strict conic fibrations
and 
$f_2$ is a fibration of relative dimension two over a toric Calabi--Yau pair $(T,B_T)$.
\end{enumerate}
Furthermore, in the second case, the crepant birational map $\phi$ only extracts log canonical places of $(X,B)$.
\end{theorem}
In~\cite{ELY25}, the authors show that a variety of complexity one 
is either toric, cluster type, or it admits a finite cover of degree at most $2$ which is cluster type.
Hence, pairs that satisfy Theorem~\ref{thm:comp2-structure}.(i) behave as varieties of complexity one.
The pairs as in Theorem~\ref{thm:comp2-structure}.(i) have alteration complexity zero, i.e., they can be transformed birationally into toric pairs after possibly taking a finite cover (see, e.g.,~\cite[Definition 2.21 and Theorem 1.4]{Mor24a}).
Therefore, we focus on understanding the pairs of complexity two in which Theorem~\ref{thm:comp2-structure}.(ii) holds. 

In the case that we drop the Fano type condition in Theorem~\ref{thm:comp2-structure} a similar statement holds; however, one needs to deal with some extra cases involving elliptic fibrations. 
Thus, in this work, we focus
on Calabi--Yau pairs of index one and complexity two
with ambient space a Fano type variety.
Strict conic fibrations are a special type of conic fibrations that behaves well with respect to the Calabi--Yau pair (see Definition~\ref{def:strict-cf}).
Strict conic fibrations are closely related to the concept of standard $\pp^1$-linkings (see~\cite[Definition 2.21]{FS20}).
To understand the pairs as in Theorem~\ref{thm:comp2-structure}, 
we want to relate the cluster type property of $(X',B')$
to that of $(X'',B'')$.
This is achieved through the following theorem.

\begin{theorem}\label{thm:ct-vs-toric-fib}
Let $f\colon (X,B)\rightarrow (Z,B_Z)$ be a strict conic fibration between Calabi--Yau pairs. Then, the pair $(X,B)$ is of cluster type if and only if $(Z,B_Z)$ is of cluster type.
\end{theorem}

Due to Lemma~\ref{lem:ct-under-bir-comp2}, Theorem~\ref{thm:comp2-structure}, and Theorem~\ref{thm:ct-vs-toric-fib}, 
to understand the cluster type property for pairs as in Theorem~\ref{thm:comp2-structure}.(ii) 
it suffices to study families of surface pairs over toric bases.

\subsection{Surface fibrations over toric varieties} 
In what follows, we study crepant fibrations $(X,B)\rightarrow (T,B_T)$ of relative dimension two where $(T,B_T)$ is a toric Calabi--Yau pair.
We focus on the case that $(X,B)$ has index one and complexity two. 
In this direction our first result is a technical one that let us simplify the structure of the general fiber by performing some suitable birational modifications over the base. 

\begin{theorem}\label{thm:reduction-to-standard-model}
Let $(X,B)$ be a Calabi--Yau pair of index one and complexity at most two. 
Let $f\colon (X,B)\rightarrow (T,B_T)$ be a crepant fibration of relative dimension two to a toric Calabi--Yau pair $(T,B_T)$.
Then, there exists a crepant birational map $\varphi\colon (X',B') \dashrightarrow (X,B)$, only extracting log canonical places of $(X,B)$, such that the commutative diagram:
\[
\xymatrix{
 (X,B) \ar[d]_{f}  & (X',B') \ar@{-->}[l]_{\varphi} \ar[d]^{f'}\\
 (T,B_T) & \ar[l]_{g} (T',B_{T'})
}
\]
satisfies the following conditions:
\begin{enumerate}
    \item[(i)] $g$ is a crepant birational morphism between toric Calabi--Yau pairs, \item[(ii)] we have ${f'}^{-1}(B_{T'})\subset B'$,
    \item[(iii)] the fibration $X'\rightarrow T'$ has $\rho(X'/T')\leq 2$ and  no degenerate divisors, and 
    \item[(iv)] for a general fiber $F'$ of $X'\rightarrow T'$ the divisor $B'|_{F'}$ is contained in the smooth locus of $F'$.
\end{enumerate}
\end{theorem}

Note that the fibration $f'\colon (X',B')\rightarrow (T',B_{T'})$ produced by Theorem~\ref{thm:reduction-to-standard-model} improves the properties of $f$.
More precisely, the pair $(T',B_{T'})$ is toric, the fibration $f'$ has relative dimension two, the pair $(X',B')$ has index one, and complexity at most two.
Theorem~\ref{thm:reduction-to-standard-model} introduces some new properties: Theorem~\ref{thm:reduction-to-standard-model}.(ii) gives a fibration of open Calabi--Yau varieties, Theorem~\ref{thm:reduction-to-standard-model}.(iii) let us avoid the existence of many birational contractions for $X'/T'$, and Theorem~\ref{thm:reduction-to-standard-model}.(iv) implies that $B'|_{F'}$ is a simple normal crossing divisor.
Fibrations that satisfy the properties mentioned above are called \emph{standard models over a toric base}
(see Definition~\ref{def:standard-model-over-toric-bcase}). 
In Lemma~\ref{lem:ct-under-bir-comp2}, we will show that $(X,B)$ is of cluster type if and only if $(X',B')$ is of cluster type.
Henceforth, Theorem~\ref{thm:reduction-to-standard-model} reduces the problem of deciding whether a surface fibration over a toric variety is of cluster type to the case of standard models over toric bases.
The main theorems of this article are the following two criteria to decide whether a standard model of relative dimension two over a toric base is of cluster type.
The criteria are slightly different depending on the relative Picard rank of the standard model. 
First, we present the theorem in the case of relative Picard rank two.

\begin{theorem}\label{thm:ct-standard-model-Pic-2}
Let $(X,B)$ be a Calabi-Yau pair of index one and complexity two. 
Let $f\colon (X,B)\rightarrow (T,B_T)$ be a standard model of relative dimension two and relative Picard rank two.
Let $(F,B_F)$ be the general fiber of $f$.
Then, the pair $(X,B)$ is of cluster type if and only if the following three conditions are satisfied:
\begin{enumerate}
    \item $B_F$ has a nodal point,
    \item each irreducible component of $B$ restricts to an irreducible component of $B_F $, and 
    \item a component of $B_F$ has positive self-intersection.
\end{enumerate}
\end{theorem}

Secondly, we present the theorem in the case of relative Picard rank one, which can be deduced by the previous theorem by performing a blow-up at a node on the general fiber. 

\begin{theorem}\label{thm:ct-standard-model-Pic-1}
Let $(X,B)$ be a Calabi--Yau pair of index one and complexity two.
Let $f\colon (X,B)\rightarrow (T,B_T)$ be a standard model of relative dimension two and relative Picard rank one.
Let $(F,B_F)$ be the log general fiber of $f$.
Then, the pair $(X,B)$ is of cluster type if and only if the following three conditions are satisfied:
\begin{enumerate}
    \item $B_F$ has a nodal point, 
    \item each irreducible component of $B$ restricts to an irreducible component of $B_F $, and
    \item we have $\vol(F)\geq 5$.
\end{enumerate}
\end{theorem}

It is a classic fact that, under some mild conditions, a del Pezzo fibration over $\pp^1$ is rational when the volume of a general fiber is at least $5$ (see, e.g.,~\cite{Pro18}).
Theorem~\ref{thm:ct-standard-model-Pic-1} provides a similar result within the theory of cluster type pairs. 
To prove Theorem~\ref{thm:ct-standard-model-Pic-2} and Theorem~\ref{thm:ct-standard-model-Pic-1}, we will use the following statement regarding the invariance of the coregularity for families of Calabi--Yau pairs.

\begin{theorem}\label{thm:invariance-coreg}
Let $f\colon (X,B)\rightarrow (T,B_T)$ be a crepant fibration to a toric Calabi--Yau pair. Then, there exists a non-trivial open subset $U\subseteq T$
for which ${\rm coreg}(X_t,B_t)={\rm coreg}(X,B)$ for all $t\in U$.
\end{theorem}

In Section~\ref{sec:ex-quest}, we show some explicit examples of 
Theorem~\ref{thm:ct-standard-model-Pic-1}
and 
Theorem~\ref{thm:ct-standard-model-Pic-2}.
Further, we show that the toricity of the base pair in
Theorem~\ref{thm:invariance-coreg} 
are indeed necessary. 

\subsection*{Acknowledgements}

The authors would like to thank Joshua Enwright and Fernando Figueroa for many useful discussions and comments.

\section{Preliminaries}

We work over an algebraically closed field $\kk$ of characteristic zero.
In this section, we recall some basic definitions and statements related to 
the complexity, cluster type pairs,
and strict conic fibrations.

\subsection{Calabi--Yau and Fano}
In this subsection, we recall some definitions regarding Calabi--Yau pairs and Fano type pairs.

\begin{definition}
{\em  
Let $X$ be a normal variety and let $X\rightarrow Z$ be a projective morphism.
We say that $X$ is of {\em Fano type} over $Z$ if there exists a boundary divisor $\Delta$ in $X$ for which $(X,\Delta)$ is klt and $-(K_X+\Delta)$ is nef and big over $Z$.
In the previous case, we also say that $X\rightarrow Z$ is a {\em Fano type morphism}.
}
\end{definition}

One of the important properties of Fano type morphisms is that they are relative Mori dream spaces. 

\begin{definition}
{\em 
Let $X$ be a normal variety and let $X\rightarrow Z$ be a projective morphism.
Let $B$ be a boundary divisor on $X$.
We say that $(X,B)$ is {\em Calabi--Yau over $Z$} if $(X,B)$ has log canonical singularities and $K_X+B\sim_{\qq,Z} 0$.
In the previous case, we may also say that $X$ is of {\em Calabi--Yau type} over $Z$.
In the case that $B$ is not a boundary divisor, we say that $(X,B)$ is a {\em Calabi--Yau sub-pair} if $(X,B)$ is a log canonical sub-pair and $K_X+B\sim_{\qq,Z} 0$.
}    
\end{definition}

\subsection{Crepant maps}
Throughout the article, we will often say that certain birational maps, finite morphisms, and fibrations are {\em crepant}.
Roughly speaking this means that the pull-back of the canonical divisor is the canonical divisor. We explain the details below in the following three definitions.

\begin{definition}
{\em
Let $\phi\colon Y\dashrightarrow X$ be a birational map between normal quasi-projective varieties. Let $B$ be a divisor on $X$. 
Let $p\colon Z\rightarrow Y$ and $q\colon Y\rightarrow X$ be a resolution of the indeterminancy of $\phi$. 
Let $B_Y$ be the unique divisor on $Y$ that satisfies:
\[
p^*(K_Y+B_Y)=q^*(K_X+B),
\]
where the previous holds as equality of divisors. 
We say that $(Y,B_Y)$ is the {\em pull-back} of the pair $(X,B)$ to $Y$.
In this case, we also say that $\phi\colon (Y,B_Y)\dashrightarrow (X,B)$ is a {\em crepant birational map}.
}
\end{definition}

\begin{definition}
{\em 
Let $f\colon Y\rightarrow X$ be a finite morphism between normal varieties.
Let $B_Y$ and $B$ be divisors on $Y$ and $X$, respectively.
We say that $f\colon (Y,B_Y)\rightarrow (X,B)$ is a {\em finite crepant morphism}
if $f^*(K_X+B)=K_Y+B_Y$ holds.
}    
\end{definition}

\begin{definition}
{\em 
Let $f\colon X\rightarrow Y$ be a fibration between normal varieties.
Let $(X,B)$ be a Calabi--Yau sub-pair over $Y$. Let $(Y,B_Y)$ be the pair induced by the canonical bundle formula on $Y$. 
Then, we say that $f\colon (X,B)\rightarrow (Y,B_Y)$ is a {\em crepant fibration}.
}
\end{definition}

\subsection{Cluster type pairs} 
In this subsection, we introduce the concept of cluster type pairs and recall some of its main properties.

\begin{definition}\label{def:cluster-type}
{\em 
Let $(X,B)$ be a Calabi--Yau pair.
We say that $(X,B)$ is of {\em cluster type} if there exists a crepant birational map
\[
\phi\colon (\pp^n,H_0+\dots+H_n) \dashrightarrow (X,B)
\]
such that 
\begin{equation}\label{eq:ct-equation} 
{\rm codim}_{\pp^n}\left(
\overline{\mathbb{G}_m^n\cap {\rm Ex}(\phi)}\right) \geq 2. 
\end{equation} 
To simplify the notation, for $n\geq 1$, we set $\Sigma^n:=H_0+\dots+H_n$ for the sum of the coordinate hyperplanes in the projective space of dimension $n$.
}
\end{definition}

For instance, any toric Calabi--Yau pair 
is of cluster type.
In the toric case, we have $X\setminus B=\mathbb{G}_m^n$.
In the case that $X$ is a $n$-dimensional $\qq$-factorial Fano variety and $(X,B)$ is of cluster type, the open subvariety $X\setminus B$ is covered by $n$-dimensional algebraic tori up to a subset of codimension at least two (see, e.g.,~\cite[Theorem 1.3.(4)]{EFM24}).
Below, we recall some properties of cluster type pairs under birational transformations.

\begin{lemma}\label{lem:ct-higher-model}
Let $(X,B)$ be a cluster type pair.
Let $\psi\colon (Y,B_Y)\rightarrow (X,B)$ be a crepant birational morphism that only extracts log canonical places and canonical places of $(X,B)$.
Then, the pair $(Y,B_Y)$ is also of cluster type.
\end{lemma}

\begin{proof}
Assume that $(X,B)$ is of cluster type.
By the condition on the extraction of $(Y,B_Y)\rightarrow (X,B)$, we know that $(Y,B_Y)$ is a Calabi--Yau pair.
Since $(X,B)$ is of cluster type, there exists a crepant birational map $\phi\colon (\pp^n,\Sigma^n)\dashrightarrow (X,B)$
such that inequality~\eqref{eq:ct-equation} holds. 
Then, the crepant birational map
\[
\xymatrix{
(\pp^n,\Sigma^n)\ar@{-->}[r]^-{\phi} & 
(X,B)\ar@{-->}[r]^-{\psi^{-1}} & 
(Y,B_Y) 
}
\]
satisfies that
\[
{\rm codim}_{\pp^n}
\left(
\overline{  
\mathbb{G}_m^n
\cap
{\rm Ex}(\psi^{-1}\cap \phi) 
}
\right) 
\geq 2.
\]
This implies that $(Y,B_Y)$ is a cluster type pair. 
\end{proof}

The following lemma is a characterization of cluster type varieties
coming from the MMP. 
The proof of the lemma follows from~\cite[Lemma 2.29]{JM24}.

\begin{lemma}\label{lem:MMP-char-ct}
Let $(X,B)$ be a Calabi--Yau pair. 
Then, the pair $(X,B)$ is of cluster type if and only if there exists a dlt modification $(Y,B_Y)\rightarrow (X,B)$ and an effective divisor $E_Y$ on $Y$ satisfying the following. The pair $(Y,B_Y+\epsilon E_Y)$ is dlt for $\epsilon>0$ small enough and the $(K_Y+B_Y+\epsilon E_Y)$-MMP terminates in a toric Calabi--Yau pair after contracting all the components of $E_Y$. 
\end{lemma}

\subsection{Complexity}
In this subsection, we recall the concept of complexity of a Calabi--Yau pair.
We state some results regarding the geometry of pairs of small complexity.  

\begin{definition}
{\em 
Let $(X,B)$ be a log pair.
The {\em complexity} of the pair, denoted by $c(X,B)$, is defined to be:
\[
\dim(X)+\dim_\qq {\rm Cl}_\qq(X)-|B|,
\]
where $|B|$ denotes the sum of the coefficients of $B$.
A {\em decomposition} $\Sigma$ of $B$ is an expression of the form $\sum_{i=1}^k b_iB_i$, where each $b_i$ is a non-negative real number and each $B_i$ is an effective Weil divisor with $\sum_{i=1}^k b_iB_i\leq B$.
The {\em rank} of $\Sigma$, denoted by $\rho(\Sigma)$ is the dimension of the $\qq$-vector subspace of ${\rm Cl}_\qq(X)$ generated by the divisors $B_i$.
The {\em norm} of $\Sigma$, denoted by 
$|\Sigma|$ is defined to be $\sum_{i=1}^k b_i$.

Let $\Sigma$ be a decomposition of $B$.
The {\em fine complexity} of $(X,B;\Sigma)$, denoted by $c(X,B;\Sigma)$, is defined to be:
\[
\dim(X)+\rho(\Sigma)-|\Sigma|.
\]
The {\em fine complexity} $\overline{c}(X,B)$ of the pair 
$(X,B)$ is the infimum among of $c(X,B;\Sigma)$ where the infimum runs over all the decompositions $\Sigma$ of $B$.
}
\end{definition}

The following theorem is a more technical version of Theorem~\ref{thm:comp1-ct} proved in~\cite{ELY25}.

\begin{theorem}\label{thm:complexity-one-implies-cluster-type}
Let $(X,B)$ be a Calabi--Yau pair of index one. 
Let $S\leq B$ be a reduced divisor.
Assume that $\overline{c}(X,S)\leq 1$.
Then, the pair $(X,B)$ is of cluster type.
\end{theorem}

We finish this subsection, showing that 
for Calabi--Yau pairs of index one and complexity two, the cluster type condition behaves well under birational transformations.

\begin{lemma}\label{lem:ct-under-bir-comp2}
Let $(X,B)$ be a Calabi--Yau pair of index one and complexity two. 
Let $\phi\colon (X',B')\dashrightarrow (X,B)$ be a crepant birational map that only extracts log canonical places of $(X,B)$.
Then, the pair $(X,B)$ is of cluster type if and only if $(X',B')$ is of cluster type.
\end{lemma}

\begin{proof}
First, assume that $(X,B)$ is of cluster type.
Since $\phi$ only extracts log canonical places, we have $c(X',B')\leq 2$.
If $\phi^{-1}$ extacts a canonical place of $(X',B')$, then we get 
$c(X',B')<c(X,B)$. 
As $(X',B')$ has index one, we conclude that $c(X',B')\leq 1$ and so $(X',B')$ is of cluster type.
Thus, we may assume that $\phi$ only extracts or contracts log canonical places. 
Therefore, if $\psi\colon (\pp^n,\Sigma^n)\dashrightarrow (X,B)$ is a crepant birational map satisfying the inequality~\eqref{eq:ct-equation}, then  so does $\psi\circ \phi^{-1}$.

Now, we assume that $(X',B')$ is of cluster type.
Let $\psi\colon (\pp^n,\Sigma^n)\dashrightarrow (X',B')$ be a crepant birational map satisfying the inequality~\eqref{eq:ct-equation}.
As $\phi$ does not contract any divisor of $X'\setminus B'$, we conclude that
$\phi\colon \psi \colon \pp^n\dashrightarrow X$ does not contract any divisor of the algebraic torus.
Therefore, the pair $(X,B)$ is of cluster type.
\end{proof}

\subsection{Strict conic fibrations}
In this subsection, we recall the concept of strict conic fibrations.

\begin{definition}\label{def:strict-cf}
{\em 
Let $f\colon X\rightarrow Z$ be a fibration. Let $(X,B)$ be a log pair.
We say that $f\colon (X,B)\rightarrow Z$ is a {\em strict conic fibration} if the following conditions are satisfied:
\begin{enumerate}
    \item[(i)] The morphism $f$ is a conic fibration of relative Picard rank one, 
    \item[(ii)] the pair $(X,B)$ is Calabi--Yau over $Z$, 
    \item[(iii)] the divisor $\lfloor B\rfloor$ has two irreducible components $S_0$ and $S_\infty$ that are horizontal over $Z$, and
    \item[(iv)] the divisors $S_0$ and $S_\infty$ are disjoint.
\end{enumerate}
}
\end{definition}

\begin{lemma}\label{lem:comp-under-p1-link}
Let $f\colon (X,B) \rightarrow (Z,B_Z)$ be a crepant fibration.
Assume that $f$ is a strict conic fibration between Calabi--Yau pairs. Then, the following statements hold:
\begin{enumerate}
    \item ${\rm coreg}(X,B)={\rm coreg}(Z,B_Z)$, 
    \item ${\rm index}(X,B)\geq {\rm index}(Z,B_Z)$, and 
    \item $c(X,B)\geq c(Z,B_Z)$.
\end{enumerate}
\end{lemma}

\begin{proof}
Let $S_0\leq \lfloor B\rfloor$ be as in Definition~\ref{def:strict-cf}.
Let $(S_0,B_{S_0})$ be the pair obtained by adjunction of $(X,B)$ to $S_0$.
Then, we have
\[
{\rm coreg}(X,B)={\rm coreg}(S_0,B_{S_0})
\]
due to the equidimensionality of dual complexes of Calabi--Yau pairs (see, e.g.,~\cite[Theorem 1.5]{FS20}).
On the other hand, the morphism $S_0\rightarrow Z$ is birational and so
\[
{\rm coreg}(S_0,B_{S_0})={\rm coreg}(Z,B_Z).
\]
This proves the first item.

If $m(K_X+B)\sim 0$, then $m(K_{S_0}+B_{S_0})\sim 0$ by adjunction.
Since $S_0\rightarrow Z$ is birational, we conclude that $m(K_Z+B_Z)\sim 0$. 
This shows the second item.

Finally, note that every component of $B$ is either $S_0$, $S_\infty$, or it is vertical over $Z$.
For every $P\in X$ which is vertical over $Z$, we have ${\rm coeff}_P(B) \leq 
{\rm coeff}_{P_Z}(B_Z)$ where $P_Z$ is the image of $P$ in $Z$.
Note that $\dim(Z)=\dim(X)-1$ and
$\dim_\qq {\rm Cl}(Z)_\qq =\dim_\qq {\rm Cl}_\qq (X)-1$. Thus, we can compute:
\[ 
c(X,B) = 
\dim X  + \dim_\qq {\rm Cl}_\qq (X) - |B| 
\geq \dim Z +1 + \dim_\qq{\rm Cl}_\qq(Z) - (|B_Z|-2)  = c(Z,B_Z).
\]
This finishes the proof of the third item.
\end{proof}

\subsection{Standard models} 
In this subsection, we introduce the concept of standard models over toric bases. 

\begin{definition}
{\em 
Let $f\colon (X,B)\rightarrow (T,B_T)$ be a crepant fibration from a Calabi--Yau pair $(X,B)$.
We say that $f$ is a {\em standard model over a toric base} if the following conditions are satisfied: 
\begin{enumerate}
\item the Calabi--Yau pair $(T,B_T)$ is toric, 
\item we have $f^{-1}(B_T)\subseteq B$, 
\item the fibration $f$ has no degenerate divisors and we have $\rho(X/T)\leq \dim(X/T)$, and
\item for a general fiber $F$ of $X \to T$ the divisor $B|_F$ is contained in the smooth locus of $F$.
\end{enumerate}
}
\label{def:standard-model-over-toric-bcase}
\end{definition}

\section{Reduction to surfaces over toric varieties}

In this section, we study Calabi--Yau pairs of index one
and complexity two. We prove a structural theorem relating 
surface fibrations over toric varieties to such pairs.
We will need the following lemma regarding conic fibrations.

\begin{lemma}\label{lem:p1-fibrations}
Let $G$ be a finite group.
Let $(X,B)$ be a $G$-equivariant Calabi--Yau sub-pair. 
Let $\phi\colon (X,B)\rightarrow (Z,B_Z)$ be a $G$-equivariant
crepant fibration and let $G_Z$ be the quotient of $G$ acting on $Z$.
Assume that the general fiber of $\phi$ is isomorphic to $\pp^1$, 
the divisor $\lfloor B\rfloor$ has two irreducible components that dominate $Z$, 
the divisor $B^{<0}$ is vertical over $Z$, and $B_Z$ is effective. 
Then, there is a commutative diagram as follows:
\[
\xymatrix{
(X,B) \ar[d]_-{\phi} & (X',B')\ar[d]^-{\phi'}\ar@{-->}[l]_-{\pi} \\
(Z,B_Z) & (Z',B_{Z'})\ar[l]_-{\pi_Z}
}
\]
satisfying the following conditions:
\begin{enumerate}
    \item $(X',B')$ is a Calabi--Yau pair, 
    \item $\pi$ is a $G$-equivariant crepant birational map only extracting log canonical places of $(X,B)$,
    \item $\pi_Z$ is a $G_Z$-equivariant crepant birational morphism only extracting log canonical places of $(Z,B_Z)$, 
    \item $\phi'$ is a $G$-equivariant crepant fibration of relative $G$-Picard rank one, 
    \item we have ${\phi'}^{-1}(\lfloor B_{Z'}\rfloor) \subset \lfloor B'\rfloor$, and
    \item $X'$ is of Fano type provided that $Z$ is of Fano type. 
\end{enumerate}
\end{lemma}

\begin{proof}
Let $r\colon Y\rightarrow (X,B)$ be a $G$-equivariant log resolution. 
Let $(Y,B_Y)$ be the pull-back of $(X,B)$ to $Y$.
For every $G_Z$-prime divisor $P$ of $Z$ with $0\leq b={\rm coeff}_P(B_Z)$  
we may find a $G$-prime divisor $P_Y$ of $Y$ that maps to $P$
and for which $b={\rm coeff}_{P_Y}(B_Y)$.
The previous follows as $\lfloor B\rfloor$ has a component that maps birationally to $Z$ and the fact that $B_Z$ is effective.
We set $D_Y:=\sum_{P\subset Z} {\rm coeff}_{P_Y}(B_Y)P_Y$. 
Let $\Gamma_Y$ be the divisor obtained from $B_Y$ by increasing to one the coefficients of all the $r$-exceptional divisors that are not contained in the support of $D_Y$. 
Then, the pair $(Y,\Gamma_Y)$ is a $G$-equivariant dlt pair
for which $K_Y+\Gamma_Y$ is $\qq$-linearly equivalent to a $G$-invariant effective divisor which is degenerate over $Z$.
Thus, we can run a $G$-equivariant $(K_Y+\Gamma_Y)$-MMP over $Z$ that terminates after contracting all the components of $\Gamma_Y-B_Y$ (see, e.g.,~\cite[Lemma 2.10]{Lai11}).
Let $X_0$ be the resulting model and $B_0$ be the push-forward of $B_Y$ to $X_0$.
Then, we have a commutative diagram:
\[
\xymatrix{
(X,B)\ar[rd]_-{\phi} \ar@{-->}[rr]^-{\pi_0} & & (X_0,B_0)\ar[ld]^-{\phi_0} \\
& (Z,B_Z) & 
}
\]
satisfying the following conditions:
\begin{enumerate}
    \item[(i)] the pair $(X_0,B_0)$ is a $G$-equivariant Calabi--Yau pair, 
    \item[(ii)] $\pi_0^{-1}$ is a $G$-equivariant crepant birational map only extracting log canonical places of $(X,B)$, 
    \item[(iii)] $\phi_0$ is a $G$-equivariant crepant fibration, 
    \item [(iv)] we have $\phi_0^{-1}(\lfloor B_Z\rfloor)\subset \lfloor B_0\rfloor$, 
    \item[(v)] all the degenerate divisors of $\phi_0$ are contained in $\lfloor B_0\rfloor$, and 
    \item[(vi)] $\phi_0$ has multiplicity one at the generic point of divisors not contained in $\lfloor B_Z\rfloor$.
\end{enumerate}
Condition (iv) and (v) hold as in the $(K_Y+\Gamma_Y)$-MMP over $T$ we contract all the degenerate divisors of $Y$ over $T$ except possibly for those contained in $B_Y^{=1}$.
Condition (vi) holds as whenever $b={\rm coeff}_P(B_Z)$ 
we can find a $G$-prime divisor $P_0$ in $X_0$ for which $b={\rm coeff}_{P_0}(B_0)$. By construction, whenever $b<1$ the $G$-prime divisor $P_0$ equals the fiber over $P$.

Now, we run a $G$-equivariant $K_{X_0}$-MMP over $Z$.
This MMP terminates with a $G$-equivariant Mori fiber space 
$X' \rightarrow Z'$ where $Z'\rightarrow Z$ is a $G_Z$-equivariant projective birational morphism. 
We obtain a commutative diagram as follows:
\[
\xymatrix{
(X_0,B_0) \ar@{-->}[r]^-{\pi'}
\ar[d]_-{\phi_0} & (X',B')\ar[d]^-{\phi'} \\
(Z,B_{Z}) & (Z',B_{Z'})\ar[l]_-{\pi_Z}
}
\]
satisfying the following conditions:
\begin{enumerate}
    \item[(a)] the pair $(X',B')$ is a $G$-equivariant Calabi--Yau pair,
    \item[(b)] $\pi'$ is a $G$-equivariant birational contraction, 
    \item[(c)] $\pi_Z$ is a $G_Z$-equivariant crepant birational morphism only extracting log canonical places of $(Z,B_Z)$, 
    \item[(d)] $\phi'$ is a $G$-equivariant crepant fibration
    of relative $G$-Picard rank one, 
    \item[(e)] we have ${\phi'}^{-1}(\lfloor B_{Z'}\rfloor)\subset \lfloor B'\rfloor$, and 
    \item[(f)] $\phi'$ has multiplicity one at the generic point of divisors not contained in $\lfloor B_{Z'} \rfloor$. 
\end{enumerate}
Conditions $(a),(b),(d)$, and $(e)$
follow from the construction of the model $X'$.
Condition $(f)$ holds as the same holds for $\phi_0$ due to (vi).
We argue condition $(c)$. If $P'$ is a prime divisor on $Z'$ such that 
$\pi_Z(P')$ has codimension at least two, then $P_0:={{\pi'}^{-1}}_*({\phi'}^*P')$
is a degenerate divisor for $\phi_0$, therefore $P_0$ appears with coefficient one in $B_0$ and so $P'$ appears with coefficient one in $B_{Z'}$.

Finally, it suffices to argue that $X'$ is of Fano type provided that $Z$ is of Fano type. In this case, the variety $Z'$ is also of Fano type (see, e.g.,~\cite[Lemma 2.17.(3)]{Mor24a}).
Therefore, we can find $\Delta_{Z'}$ such that $(Z',\Delta_{Z'})$ is klt and $-(K_{Z'}+\Delta_{Z'})$ is ample.
Further, we may assume that the coefficients of $\Delta_{Z'}$ are arbitrarily close to one along the support of $\lfloor B_{Z'}\rfloor$.
The previous statement is obtained by possibly taking combinations of $B_{Z'}$ and $\Delta_{Z'}$.
Further, we may assume that $\Delta_{Z'}$ is $G$-invariant.
Let $(Z',D_{Z'})$ be the klt pair obtained by the canonical bundle formula applied to $(X',\lfloor B'\rfloor_{\rm hor})$ with respect to $\phi'$.
By condition (f), we know that the support of $D_{Z'}$ is contained in $\lfloor B_{Z'}\rfloor$, so we may assume that $\Delta_{Z'}\geq D_{Z'}$.
Let $F_{Z'}:=\Delta_{Z'}-D_{Z'}$.
Then, for $\epsilon>0$ small enough the $G$-equivariant divisor 
\[
-(K_{X'}+(1-\epsilon)\lfloor B'\rfloor_{\rm hor}+{\phi'}^*F_{Z'}) 
\equiv 
- {\phi'}^*(K_{Z'}+\Delta_{Z'})+\epsilon \lfloor B'\rfloor_{\rm hor}
\]
is ample and the pair $(X',(1-\epsilon)\lfloor B'\rfloor_{\rm hor}+{\phi'}^*F_{Z'})$ is klt.
This implies that $X'$ is of Fano type.
\end{proof}

Before proving the structural theorem for pairs of complexity two, we will need a lemma about the behavior of the cluster type property under conic fibrations.

\begin{lemma}\label{lem:conic-fib-ct}
Let $\phi\colon (X,B)\rightarrow (Z,B_Z)$ be a crepant fibration of relative dimension one.
Assume that the general fiber of $\phi$ is isomorphic to $\pp^1$
and that the divisor $\lfloor B\rfloor$ has two irreducible components that dominate $Z$.
If $(Z,B_Z)$ is of cluster type, then $(X,B)$ is of cluster type. 
\end{lemma}

\begin{proof}
Since $(Z,B_Z)$ is of cluster type, there exists a toric Calabi--Yau pair
$(T,B_T)$ and a crepant birational map
$\pi_T \colon (Z,B_Z)\dashrightarrow (Z,B_Z)$ that only extracts log canonical places of $(Z,B_Z)$.
Let $X'\rightarrow X$ be a log resolution of $(X,B)$ for which $X'$ dominates $T$.
Let $(X',B')$ be the pull-back of $(X,B)$ to $X'$. Note that, a priori, the divisor $B'$ may have negative coefficients, however, the prime components with negative coefficients are vertical over $T$. Therefore, we can apply Lemma~\ref{lem:p1-fibrations}
to the crepant fibration $(X',B')\rightarrow (T,B_T)$.
We obtain a diagram as follows:
\[
\xymatrix{ 
(X,B)\ar[d]_-{\phi} & & (Y,B_Y)\ar@{-->}[ll]_-{\pi}\ar[d]^-{\phi_Y} \\
(Z,B_Z)\ar@{-->}[r]^-{\pi_T} & (T,B_T) &
(T',B_{T'})\ar[l]_-{\pi_{T'}}
}
\]
where $\phi_Y$ is a Mori fiber space, 
$\pi$ is a crepant birational map, and 
$\pi_{T'}$ is a toric birational morphism.
In particular, the pair $(T',B_{T'})$ is a toric Calabi--Yau pair.
First, we argue that $(Y,B_Y)$ is a toric pair. 
Indeed, we have $\phi^{-1}(B_{T'})\subset \lfloor B_Y\rfloor$
and $\lfloor B_Y \rfloor$ has two prime components that dominate $T'$.
Therefore, we can compute:
\[
c(Y,B_Y)=c(T',B_{T'})+\dim(Y/T')+\rho(Y/T')-|{B_{Y}}_{\rm hor}| = 0.
\]
Hence, the pair $(Y,B_Y)$ is toric.
Now, we argue that $\pi$ only extracts log canonical places of $(X,B)$.
Indeed, if $S$ is a prime divisor on $Y$ whose center on $X$ has higher codimension, then $\phi_Y(S)\subset B_{T'}$.
As $\phi_Y^{-1}(B_{T'})\subset \lfloor B_Y \rfloor$, we conclude that $S$ appears with coefficient one in $B_Y$.
Hence, the pair $(X,B)$ is of cluster type.
\end{proof}

We turn to proving a structural theorem for Calabi--Yau pairs of index one and complexity at most two.

\begin{proof}[Proof of Theorem~\ref{thm:comp2-structure}]
If the complexity of the pair $(X,B)$ is zero, then $(X,B)$ is a toric Calabi--Yau pair, so condition (i) holds for $(X,B)=(X',B')=(Y,B_Y)$.
If the complexity of the pair $(X,B)$ is one, then we know that $(X,B)$ is of cluster type due to Theorem~\ref{thm:comp1-ct}.
Therefore, condition (i) again holds by taking $(X,B)=(X',B')=(Y,B_Y)$.
Thus, from now on, we focus on Calabi--Yau pairs of index one and complexity two.
We assume that $n\geq 2$.

Let $(X,B)$ be a Calabi--Yau pair of
dimension $n$, index one, and complexity two. We will proceed by induction on $n$.
In the case that $n=1$, the variety $X$ is isomorphic to $\pp^1$ as $X$ is of Fano type. Moreover, the divisor $B$ equals $\{0\}+\{\infty\}$ as $(X,B)$ is log canonical of index one.
Therefore, it suffices to take
$(Y,B_Y)=(X',B')=(X,B)$ and condition (i) holds.

By~\cite[Theorem 5.6]{Mor24a}, there exists a crepant birational map
$(X_0,B_0)\dashrightarrow (X,B)$, only extracting log canonical places
of $(X,B)$, and a tower of Mori fiber spaces
\[
\xymatrix{
(X_0,B_0)\ar[r]^-{f_1} & (X_1,B_1)\ar[r]^-{f_2} & \dots  \ar[r]^-{f_k} &  (X_k,B_k) \ar[r]^-{f_{k+1}} & {\rm Spec}(\kk),
}
\]
of which at least $n-2$ of the $f_i$'s are strict conic fibrations.
Note that either $k=n$ or $k=n-1$.
Otherwise, we get $\dim X_0 \leq n-2$, leading to a contradiction. 
Assume that $k=n-1$. Then, by dimension reasons, one of the $f_i$'s, let's say $f_j$, must be a surface fibration. 
In this case, it suffices to set $(X',B'):=(X_0,B_0)$
and $(X'',B''):=(X_j,B_j)$.
Indeed, by~\cite[Theorem 1.5]{Mor24a}, the pair $(X_{j+1},B_{j+1})$
is a toric Calabi--Yau pair.
Thus, condition (ii) holds. 

From now on, we assume that $k=n$ and then all the $f_i$'s are conic fibrations.
Furthermore, at least $n-2$ of them are strict conic fibrations.
We prove by induction on the dimension that condition (i) holds.
We proceed in two cases depending on the number of irreducible components of $B_0$ that are horizontal over $X_1$.\\

\textit{Case 1:} Assume that $B_0$ has two components that are horizontal over $X_1$.\\ 

In this case, by the canonical bundle formula, we know that 
$(X_1,B_1)$ is a Calabi--Yau pair of index one and complexity two.
By induction on the dimension
there exists a crepant birational model $(X_1',B_1')$ of $(X_1,B_1)$
that admits a finite crepant morphism $(Y_1,B_{Y_1})\rightarrow (X'_1,B_1')$
of degree at most $2$ such that $(Y_1,B_{Y_1})$ is of cluster type.
By Lemma~\ref{lem:p1-fibrations}, there is a crepant birational model
$(X_0',B_0')$ of $(X_0,B_0)$, with $B_0'$ effective, such that $X_0'$
admits a Mori fiber space to $X_1'$.
In the previous line, we may need to replace both $(X'_1,B'_1)$ and $(Y,B_Y)$ by higher models only extracting log canonical places. 
Due to Lemma~\ref{lem:ct-higher-model}, the cluster type property of $(Y,B_Y)$ is preserved under this replacement.
Let $Y_0$ be the normalization of the main component of
$X_0' \times_{X_1'} Y_1$. 
Let $(Y_0,B_{Y_0})$ be the sub-pair induced by pull-back of $(X_0',B_0')$ to $Y_0$. 
Note that $(Y_0,B_{Y_0})$ is a $\zz_2$-equivariant Calabi--Yau sub-pair over $Y_1$ and $\lfloor B_{Y_0}\rfloor$ has two irreducible components that dominate $Y_1$.
By Lemma~\ref{lem:p1-fibrations}, possibly replacing $(Y_1,B_{Y_1})$ with a higher birational model that only extracts log canonical places,
we may assume that $(Y_0,B_{Y_0})$ is a $\zz_2$-equivariant Calabi--Yau pair.
By Lemma~\ref{lem:ct-higher-model}, the cluster type condition of $(Y,B_Y)$ is preserved in the previous replacement.
Thus, we get a commutative diagram as follows:
\[
\xymatrix{
(X_0,B_0) \ar[d]^-{f_1} & (X_0',B_0')\ar[d]^-{f'_1} \ar@{-->}[l]_-{\phi} & (Y_0,B_{Y_0})\ar[l]_-{/\zz_2}\ar[d]^-{f_Y}\\
(X_1,B_1) & (X_1',B_1')\ar@{-->}[l]_-{\phi_1} & (Y_1,B_{Y_1})\ar[l]_-{/\zz_2}
}
\]
where all the vertical morphisms are conic fibrations.
Furthermore, the pair $(X_0',B_0')$ is a crepant birational model of $(X_0,B_0)$
and $(Y_0,B_{Y_0})\rightarrow (X_0',B_0')$ is a crepant finite morphism
of degree at most $2$. 
If suffices to show that $(Y_0,B_{Y_0})$ is of cluster type.
By construction, we know that $B_{Y_0}$ has two irreducible components that dominate $Y_1$.
Therefore, the pair $(Y_0,B_{Y_0})$ is of cluster type by Lemma~\ref{lem:conic-fib-ct} and so condition (ii) holds in this case.\\ 

\textit{Case 2:} Assume that $B_0$ has a unique irreducible component that dominates $X_1$.\\ 

The horizontal component of $B_0$ 
admits a generically $2$-to-$1$ morphism to $X_1$.
We let $Y_1\rightarrow X_1$ be the finite morphism of degree at most two obtained by applying Stein factorization to the morphism to $X_1$ from the normalization of the horizontal component of $B_0$.
Let $X_1$ be the normalization of the main component of $X_0\times_{X_1} Y_1$. Then, we have a commutative diagram:
\[
\xymatrix{
(X_0,B_0)\ar[d]^-{f_1} & (Y_0,B_{Y_0})\ar[d]^-{f_Y} \ar[l]_-{/\zz_2}\\ 
(X_1,B_1) & (Y_1,B_{Y_1}). \ar[l]_-{/\zz_2}
}
\]
Note that $|B_0|\geq \dim X_0+\rho(X_0)-2$
and so 
\[
|\lfloor B_1\rfloor| \geq \dim X_0 +\rho(X_0)-3 
= \dim X_1 + \rho(X_1) -1.
\]
Hence we get $c(X_1,\lfloor B_1\rfloor)\leq 1$.
Thus, we have $\overline{c}(Y_1,\lfloor B_{Y_1}\rfloor)\leq 1$.
A priori the divisor $B_{Y_0}$ may not be effective, 
however, applying Lemma~\ref{lem:p1-fibrations}, we may assume that $B_{Y_0}$ is effective after possibly replacing $(Y_1,B_{Y_1})$ with a higher model that only extracts log canonical places. 
This replacement does not make the fine complexity larger (see, e.g.,~\cite[Lemma 3.32]{MS21}).
Note that $(Y_1,B_{Y_1})$ is a Calabi--Yau pair of index one. 
By Theorem~\ref{thm:complexity-one-implies-cluster-type},
we know that $(Y_1,B_{Y_1})$ is of cluster type.
Furthermore, the divisor $B_{Y_0}$ has two components that dominate $Y_1$.
By Lemma~\ref{lem:comp-under-p1-link}, we conclude that $(Y_0,B_{Y_0})$ is of cluster type and so condition (ii) holds in this case.
\end{proof}

Next, we explain how the complexity behaves
under strict conic fibrations. 
We will use the following lemma 
that allows us to deduce that a suitable conic fibration over a Fano type variety
is also a Fano type variety.

\begin{lemma}\label{lem:FT-conic-fib}
Let $\phi\colon (X,B)\rightarrow (Z,B_Z)$ be a crepant Calabi--Yau fibration of relative dimension one. Assume that the following conditions hold:
\begin{enumerate}
    \item the divisor $\lfloor B\rfloor$ has two irreducible components over $Z$, 
    \item every degenerate divisor of $\phi$ is contained in $\lfloor B\rfloor$, and 
    \item $\phi^{-1}(\lfloor B_Z\rfloor)\subset \lfloor B\rfloor$. 
\end{enumerate}
If $Z\rightarrow W$ is of Fano type, then $X\rightarrow W$ is of Fano type.
\end{lemma}

\begin{proof}
By Lemma~\ref{lem:p1-fibrations}, we know that there is a crepant birational map $\phi\colon (X',B')\dashrightarrow (X,B)$ that only extracts log canonical places of $(X,B)$ for which $X'\rightarrow W$ is of Fano type.
Let $E$ be a divisor on $X$ which is contracted on $X'$.
Assume that $E$ is not contained in the support of $\lfloor B\rfloor$, so it is not degenerate over $Z$.
As $\phi^{-1}(\lfloor B_Z\rfloor)\subset \lfloor B\rfloor$, we conclude that the image of $E_Z$ of $E$ in $Z$ is not a divisorial log canonical place of $(Z,B_Z)$. As $\phi$ only extracts log canonical places of $(X,B)$, we conclude that it does not extract a divisor that maps to $E_Z$. 
Thus, $E$ cannot be contracted in the model $X'$.
We conclude that all the divisors extracted by $\phi^{-1}$ are log canonical places of $(X,B)$. Therefore, $X$ is of Fano type over $W$ by~\cite[Lemma 2.17.(3)]{Mor24a}.
\end{proof}

Now, we turn to prove Theorem~\ref{thm:ct-vs-toric-fib} which explains how the cluster type property behaves under strict conic fibrations. 

\begin{proof}[Proof of Theorem~\ref{thm:ct-vs-toric-fib}]
If $(Z,B_Z)$ is of cluster type,
then $(X,B)$ is of cluster type by Lemma~\ref{lem:conic-fib-ct}. 

From now on, we assume that $(X,B)$ is of cluster type and we show that $(Z,B_Z)$ is of cluster type. 
Both Calabi--Yau pairs $(X,B)$ and $(Z,B_Z)$ have index one and coregularity zero by Lemma~\ref{lem:comp-under-p1-link}.
From the strict conic fibration condition, we know that 
$f^{-1}(B_Z) \subset B$
and that $f$ has no degenerate divisors.
By Lemma~\ref{lem:MMP-char-ct}, we know that there exists a dlt modification $(Y,B_Y)\rightarrow (X,B)$ and an effective divisor $E_Y$ on $Y$ satisfying the following.
The pair $(Y,B_Y+\epsilon E_Y)$ is dlt for $\epsilon>0$ small enough
and the $(K_Y+B_Y+\epsilon E_Y)$-MMP terminates in a toric Calabi--Yau pair $(T,B_T)$ after contracting all the components of $E_Y$.
We let $E_Z$ be the image of $E_Y$ in $Z$.
Since $(Y,B_Y+\epsilon E_Y)$ is dlt, the divisor $E_Y$ does not share any component with $B_Y$, therefore $E_Z$ does not share any component with $B_Z$.
Hence, we can find a dlt modification $(Z',B_{Z'})\rightarrow (Z,B_Z)$ such that the strict transform $E_{Z'}$ of $E_Z$ in $Z'$ 
does not contain any strata of $B_{Z'}$.
Thus, the pair $(Z',B_{Z'}+\epsilon E_{Z'})$ is dlt.
By Lemma~\ref{lem:FT-conic-fib}, possibly replacing 
$(Z',B_{Z'})$ with a qdlt modification $(Z'',B_{Z''})$,
we can find a Calabi--Yau pair $(Y',B_{Y'})$ crepant birational to $(Y,B_Y)$ that admits a Mori fiber space $f'\colon Y'\rightarrow Z''$.
The fibration $f'$ has no degenerate divisors and ${f'}^{-1}(B_{Z''})\subset B_{Y'}$.
Let $(Y'',B_{Y''})\rightarrow (Y',B_{Y'})$ be a higher birational model, only extracting log canonical places, such that $Y''\dashrightarrow Y$ is a small birational map.
Let $f''\colon (Y'',B_{Y''})\rightarrow (Z'',B_{Z''})$ be the corresponding crepant fibration.
We have a commutative diagram as follows:
\[
\xymatrix{
(X,B)\ar[d]_-{f} & (Y,B_Y)\ar[l] & (Y'',B_{Y''})\ar[d]^-{f''}\ar@{-->}[l]\\
(Z,B_Z) & (Z',B_{Z'})\ar[l] & (Z'',B_{Z''}).\ar[l]
}
\]
Furthermore, we know that $B_{Y''}$ has two irreducible components over $Z''$, every degenerate divisor of $f''$ is contained in $B_{Y''}$, and 
${f''}^{-1}(B_{Z''})\subset B_{Y''}$. 
Let $E_{Y''}$ be the strict transform of $E_Y$ in $Y''$. 
On the other hand, note that $F_{Y''}:={f''}^*(E_{Z''})$ has the same support as $E_{Y''}$.
Since $E_{Y''}$ is exceptional over $T$, by the negativity lemma, we have
\begin{equation}\label{eq:base-locus}
\supp(F_{Y''})=\supp(E_{Y''}) \subseteq 
{\rm Bs}_{-}(K_{Y''}+B_{Y''}+F_{Y''}).    
\end{equation}
Let $A_{Z''}$ be an ample divisor in $Z''$.
By equality~\eqref{eq:base-locus} and~\cite[Lemma 1.27]{Mor18a}, we get
\begin{equation}\label{eq:containmnet-bs}
\supp(F_{Y''})\subseteq {\rm Bs}(K_{Y''}+B_{Y''}+F_{Y''}+\epsilon {f''}^*A_{Z''})    
\end{equation}
for $\epsilon$ small enough.

We run a $(K_{Z''}+B_{Z''}+\epsilon E_{Z''})$-MMP with scaling of $A_{Z''}$.
Every contraction of this MMP is of Fano type, 
so we can apply Lemma~\ref{lem:FT-conic-fib} to perform the two-ray game as explained below.
Let $Z''\dashrightarrow Z''_1$ be the first birational contraction in this MMP. 
Assume that $Z''\dashrightarrow Z''_1$ is a flip for the divisor $K_{Z''}+B_{Z''}+\epsilon E_{Z''}$ and let $Z''\rightarrow W$ be the corresponding flipping contraction.
We know that $Z''\rightarrow W$ is a Fano type contraction so by Lemma~\ref{lem:FT-conic-fib} we conclude that $Y''\rightarrow W$ is of Fano type as well.
Therefore, we can run a $(K_{Y''}+B_{Y''}+F_{Y''})$-MMP over $W$ which terminates with a good minimal model whose ample model is $Z''_1$.
Thus, we obtain a commutative diagram as follows:
\[
\xymatrix{ 
(Y'',B_{Y''}+F_{Y''})\ar@{-->}[rr]\ar[d]^-{f''} & & (Y''_1,B_{Y''_1}+F_{Y''_1})\ar[d]^-{f''_1}\\  
(Z'',B_{Z''}+\epsilon E_{Z''})\ar[rd]\ar@{-->}[rr] & & 
(Z''_1,B_{Z''_1}+\epsilon E_{Z''_1})  \ar[ld] \\
& W. & 
}
\]
The case of divisorial contractions is analogous. By construction 
all the degenerate divisors 
of $f''_1$ are log canonical places of $(Y''_1,B_{Y''_1})$ and 
${f''_1}^{-1}(B_{Z''_1})\subset B_{Y''_1}$ holds.
Thus, we may proceed inductively and get a sequence of MMP's:
\[
\xymatrix{
(Y'',B_{Y''}+F_{Y''})\ar@{-->}[r]\ar[d]^-{f''} & (Y''_1,B_{Y''_1}+F_{Y''_1})\ar@{-->}[r]\ar[d]^-{f''_1} & (Y''_2,B_{Y''_2}+F_{Y''_2})\ar@{-->}[r]\ar[d]^-{f''_2} & \dots \\  
(Z'',B_{Z''}+\epsilon E_{Z''})\ar@{-->}[r] & 
(Z''_1,B_{Z''_1}+\epsilon E_{Z''_1})\ar@{-->}[r] &
(Z''_2,B_{Z''_2}+\epsilon E_{Z''_2})\ar@{-->}[r] & \dots \\
}
\]
where the vertical arrows are conic fibrations. 
As usual, the subscript means push-forward in the corresponding model.
By construction, we know that for each $i\in \zz_{\geq 1}$, the divisor
$K_{Z''_i}+B_{Z''_i}+\epsilon E_{Z''_i}+t_iA_{Z_i}$ is nef.
In particular, we get that
$K_{Y''_i}+B_{Y''_i}+F_{Y''_i}+t_i{f''_i}^*A_{Z''_i}$ is nef.
Since $(Z'',B_{Z''}+\epsilon E_{Z''})$ is a pseudo-effective pair, we know that $\lim_{i\rightarrow \infty}t_i=0$.
By the containment~\eqref{eq:containmnet-bs}, we conclude that for $j\gg 0$
the divisor $F_{Y''}$ is contracted in the model $Y''_j$.
Therefore, we get $c(Y''_j,B_{Y''_j})=0$ and so $(Y''_j,B_{Y''_j})$ is a toric Calabi--Yau pair.
Thus, we deduce that $(Z''_j,B_{Z''_j})$ is a toric Calabi--Yau pair. 
Hence, the pair $(Z,B_Z)$ has a dlt modification that admits a birational contraction to a toric Calabi--Yau pair
$(Z''_j,B_{Z''_j})$. We conclude that $(Z,B_Z)$ is of cluster type. 
\end{proof}

To conclude this section, we prove Theorem~\ref{thm:reduction-to-standard-model}.

\begin{proof}[Proof of Theorem~\ref{thm:reduction-to-standard-model}]
    Let $(Y,B_{Y}) \to (X,B)$ be a $\qq$-factorial dlt modification such that
    for each prime component $P$ of $B_T$, there exists a prime component $Q$ of $B_Y$ mapping onto $P$.
    Increase all the coefficients of the degenerate divisors of $Y$ over $T$ to one,
    to obtain the log canonical pair $(Y,B_Y + D_Y)$, where the components of
    $D_Y$ are the degenerate divisors of $Y$ over $T$ that are not log canonical places.
    Let $(Y,B_Y) \dashrightarrow (X_0,B_0)$ be the crepant birational transformation induced by a $(K_Y + B_Y + D_Y)$-MMP over $T$. 
    This MMP contracts all the component of $D_Y$ (see ~\cite[Lemma 2.10]{Lai11}),
    so $(X_0,B_0) \dashrightarrow (X,B)$ only extracts log canonical places of $(X,B)$,
    and therefore $c(X_0,B_0) \leq c(X,B)$. By~\cite[Lemma 3.2]{FMM22},
    we have $K_{X_0} + B_0 \sim 0$. Therefore, we have the diagram:
    \[
        \xymatrix{
        (X,B) \ar[d]_{f} & (X_0,B_0) \ar@{-->}[l] \ar[dl]^{f_0} \\
        (T,B_T)
        }
    \]
    where $(X_0,B_0) \xrightarrow{f_0} (T,B_T)$ satisfies condition (ii).
    
    Take a dlt modification $(X_1,B_1) \to (X_0,B_0)$. Then the general fiber $(F_1,B_{F_1})$
    is dlt, and by~\cite[Lemma 2.12]{GLM23}, the divisor $B_{F_1}$ is in the smooth locus of $F_1$.
    Run a $K_{X_1}$-MMP over $T$ to obtain a crepant birational map $(X_1,B_1) \dashrightarrow (X_2,B_2)$ and a Mori Fiber Space
    $\psi \colon X_2 \to Z$ over $T$. In particular, we have $\rho(X_2/Z) = 1$.
    The general fiber $F_1$ is a projective surface with canonical singularities, and because contracting a curve increases the log discrepancies,
    we get that any curve contracted by the restriction of the $K_{X_1}$-MMP to the fiber $F_1$ maps to a smooth point. 
    Therefore the general fiber $(F_2,B_{F_2})$ over $T$, and hence over $Z$, satisfies condition (iv).
    We have the following diagram:
    \[
        \xymatrix{
         & (X_1,B_1) \ar[d] \ar@{-->}[dr]_{\phi}\\
        (X,B) \ar[d]_{f} & (X_0,B_0) \ar@{-->}[l]_{\pi} \ar[dl]^{f_0} & (X_2,B_2) \ar[d]^{\psi} \\
        (T,B_T) & & (Z,B_Z) \ar[ll]^{g}
        }
    \]
    where $(Z,B_Z)$ is obtained by applying the canonical bundle formula to $\psi$ such that $g^*(K_T + B_T) = K_Z + B_Z$.
    We have two possible cases; either $\dim Z = \dim T$ or $\dim Z = \dim T + 1$.\\

    \noindent\textit{Case 1:} We assume that $\dim Z = \dim T$.\\

    Let $E\subseteq Z$ be a prime exceptional divisor over $T$. We show that $\coeff_E B_Z = 1$. Notice that, as $T$ is smooth outside $B_T$ and $g^*(K_T + B_T) = K_Z + B_Z$, with $B_Z \geq 0$, the center of $E$ in $T$ lies on $B_T$.
    Consider $\pi_*\phi^* \psi^* E$ in $X_0$. If $\pi_*\phi^* \psi^* E$ is a divisor on $X_0$, then by condition (ii) of $f_0$, the divisor
    $\pi_*\phi^* \psi^* E$ is on the boundary $B_0$, so it appears with coefficient 1. Therefore, $E$ appears in $(Z,B_Z)$ 
    with coefficient one.
    If $\pi_*\phi^* \psi^* E$ is not a divisor, then $\pi$ extracts $\phi^* \psi^* E$ with coeffcient one and the same argument follows.
    Therefore, $g$ only extracts log canonical places of $(T,B_T)$, so $c(Z,B_Z) = c(T,B_T) = 0$, and by~\cite[Theorem 1.2]{BMSZ18}, $(Z,B_Z)$ is a toric Calabi--Yau pair. Then, $\psi \colon (X_2,B_2) \to (Z,B_Z)$ satisfies conditions (i),(ii),(ii), and (iv).\\

    \noindent\textit{Case 2:} We assume that $\dim Z =\dim T +1$.\\
    
    We argue that $\rho(Z/T) = D_Z+1$ where $D_Z$ is the number of degenerate divisors on $Z$ over $T$ minus the number of codimension one components of the union of the image of the degenerate divisors.
    Indeed, assume that $g\colon Z\to T$ has no degenerate divisors. Let $\mu\colon Z \dashrightarrow Z'$ be a $K_Z$-MMP over $T$.
    Any divisor $P$ on $Z$ maps to a divisor $Q$ on $T$. Even more, because there are no degenerate divisors, $g^* Q = \lambda P$.
    Thus, $P$ is trivial over $T$, and so it is not contracted by $\mu$.
    Hence, $Z \dashrightarrow Z'$ is small, and $1 = \rho(Z'/T) = \rho(Z/T)$.
    Therefore, we conclude that $\rho(X/T)=2+D_X$ where $D_X$ is defined similarly to $D_Z$, indeed the number of degenerate divisors of $X$ over $T$ which agrees with the number of degenerate divisors of $Z$ over $T$.
    
    We now proceed to contract the degenerate divisors of $X$ over $T$
    in order to obtain a fibration of relative Picard rank $2$ and without degenerate divisors. 
    First, repeat the argument at the beginning of the proof to obtain 
    $(X_3,B_3) \to (X_2,B_2)$ that contracts all the degenerate divisors that do not map to $B_T$.
    For each component $P$ in $B_T$, there exists $Q$ in $B_3$ such that $B_3$ maps onto $P$.
    We then run a $(K_{X_3} + B_3 - Q)$-MMP, and in the resulting $(X_4,B_4)$, the prime divisor $Q$ is the unique component of $B_4$ that maps to $P$. 
    We repeat inductively this procedure for all the components of $B_T$, and we obtain $f'\colon (X',B') \to (T,B_T)$, 
    where $(X_2,B_2) \dashrightarrow (X',B')$ contracts all the degenerate divisors of $X_2$ over $T$.
    Therefore, $\rho(X'/T) = 2$, and so $f'$ satifies conditions (i),(ii),(iii), and (iv).
\end{proof}

\section{Invariance of coregularity}
In this section, we prove a statement regarding the invariance of coregularity for Calabi--Yau pairs over toric bases. First, we prove the following proposition.

\begin{proposition}\label{prop:coreg-over-p1}
Let $(X,B)$ be a Calabi--Yau pair.
Let $\pi\colon (X,B)\rightarrow (\pp^1,\{0\}+\{\infty\})$ be a crepant fibration of relative dimension $d$.
Assume that $\pi^{-1}(0)\subset \lfloor B\rfloor$ and $\pi^{-1}(\infty)\subset \lfloor B\rfloor$.
Let $V$ be an irreducible component of $\pi^{-1}(0)$ of dimension $d$.
Let $\nu\colon W\rightarrow V$ be the normalization and $(W,B_W)$ be the pair obtained from adjunction of $(X,B)$ to $W$.
Then, there exists a non-trivial open set $U\subseteq \pp^1$ for which 
\[
{\rm coreg}(X_t,B_t) = {\rm coreg}(W,B_W) = {\rm coreg}(X,B)
\]
holds for every $t\in U$.
\end{proposition}

\begin{proof}
Let $(Y,B_Y)\rightarrow (X,B)$ be a $\qq$-factorial dlt modification.
Note that, by definition of coregularity, we have
\begin{equation}\label{eq1}
{\rm coreg}(X,B)={\rm coreg}(Y,B_Y).
\end{equation}
For $t\in \pp^1$ general we have 
\begin{equation}\label{eq2}
{\rm coreg}(X_t,B_t)={\rm coreg}(Y_t,B_{Y_t})
\end{equation}
as $(Y_t,B_{Y_t})\rightarrow (X_t,B_t)$ is a dlt modification. 
Let $S_0$ and $S_\infty$ be the reduced fibers of $\pi_Y\colon Y\rightarrow \pp^1$
over $\{0\}$ and $\{\infty\}$, respectively.
Observe that $S_0,S_\infty \subset \lfloor B \rfloor$. 
Let $(S_0,B_{S_0})$ and $(S_\infty,B_{S_\infty})$
be the pairs obtained by adjunction of $(Y,B_Y)$ to $S_0$ and $S_\infty$, respectively.
Then, each pair $(S_0,B_{S_0})$ and $(S_\infty,B_{S_\infty})$ is a sdlt Calabi--Yau pair.
Let $V_{0,1},\dots,V_{0,\ell_0}$ be the irreducible components of $S_0$
and let $V_{\infty,1},\dots,V_{\infty,\ell_\infty}$ be the irreducible components of $S_\infty$. For each $V_{i,j}$ as above, we let $(V_{i,j},B_{V_{i,j}})$ be the dlt Calabi--Yau pair obtained by adjunction of $(Y,B_Y)$ to $V_{i,j}$.
The dual complex of a Calabi--Yau pair is equidimensional (see, e.g.,~\cite[Theorem 1.6]{FS20}), so we get 
\begin{equation}\label{eq3}
{\rm coreg}(S_i,B_{S_i}) = {\rm coreg}(V_{i,j},B_{V_{i,j}}) = \dots 
= {\rm coreg}(V_{i,\ell_i},B_{V_{i,\ell_i}})
\end{equation}
for each $i\in \{0,\infty\}$.
For some $j\in \{1,\dots,\ell_0\}$, we have a projective morphism 
$V_{0,j}\rightarrow W$ that we can factor as $p\colon V_{0,j} \rightarrow W'$
and $q\colon W'\rightarrow W$ where $p$ is a birational morphism
and $q$ is finite.
By~\cite[Lemma 7.4]{EJM23} the coregularity does not change under finite crepant morphisms, so we conclude that 
\begin{equation}\label{eq4}
{\rm coreg}(V_{0,j},B_{V_{0,j}})={\rm coreg}(W,B_W).
\end{equation} 
Let $Z$ be a minimal log canonical center of $(Y,B_Y)$
among the log canonical centers of $(Y,B_Y)$ that dominate $\pp^1$.
Let $(Z,B_Z)$ be the Calabi--Yau pair obtained by adjunction
of $(Y,B_Y)$ to $Z$.
We have an induced projective morphism $\pi_Z\colon Z\rightarrow \pp^1$ 
and the dimension of the fiber over a general point $t\in \pp^1$ equals $c={\rm coreg}(Y_t,B_{Y_t})$.
Note that $\pi_Z^{-1}(0)\subset \lfloor B_Z\rfloor$
and $\pi_Z^{-1}(\infty)\subset \lfloor B_Z\rfloor$ follow from adjunction. 
Let $p_Z \colon Z \rightarrow C$ and $q_Z\colon C\rightarrow \pp^1$ 
be the induced Stein factorization.
Applying the canonical bundle formula for $(Z,B_Z)$ to $C$, we obtain
a pair $(C,B_C)$ such that $q_Z^*(K_{\pp^1}+\{0\}+\{\infty\})=K_C+B_C$. 
Therefore, the finite morphism $q_Z$ only ramifies along $\{0\}$ and $\{\infty\}$.
By~\cite[Proposition 3.32]{Mor20c}, we conclude that $(C,B_C)\simeq (\pp^1,\{0\}+\{\infty\})$.
Therefore, every log canonical center of $(Z,B_Z)$ maps to either $\{0\}$ or $\{\infty\}$ via $p_Z$. By~\cite[Theorem 1.1]{FS20}, we conclude that $(Z,B_Z)$ is a plt pair.
In particular, we get $\lfloor B_Z\rfloor=\pi_Z^{-1}(0)+\pi_Z^{-1}(\infty)$.
Henceforth, the subvariety $\pi^{-1}_Z(0)$ has dimension $c$ and it is a minimal log canonical center of $(V_{0,k},B_{V_{0,k}})$ for some suitable $k\in \{1,\dots,\ell_0\}$.
Note that a minimal log canonical center of the dlt pair $(V_{0,k},B_{V_{0,k}})$
is also a minimal log canonical center of $(Y,B_Y)$.
By equalities~\eqref{eq3} and~\eqref{eq4}, we have
\[
{\rm coreg}(V_{0,k},B_{V_{0,k}})=
{\rm coreg}(V_{0,j},B_{V_{0,j}})=
{\rm coreg}(W,B_W).
\]
On the other hand, by equalities~\eqref{eq1} and~\eqref{eq2}, we have
\[
{\rm coreg}(X_t,B_t)=
{\rm coreg}(Y_t,B_{Y_t})=
{\rm coreg}(V_{0,k},B_{V_{0,k}})=
{\rm coreg}(Y,B_Y)=
{\rm coreg}(X,B).
\]
Thus, the two claimed equalities hold.
\end{proof}

Now, we turn to prove Theorem~\ref{thm:invariance-coreg} in the case of higher-dimensional toric bases.

\begin{proof}[Proof of Theorem~\ref{thm:invariance-coreg}]
First note that the statement is known if $\dim T=1$ as it follows from
Proposition~\ref{prop:coreg-over-p1}.
Therefore, we can proceed by induction on the dimension of the base $T$.
As above, we denote by $d$ the relative dimension of the fibration $X\rightarrow T$.

By performing a crepant birational modification of $(X,B)$ over $T$, we may assume that $f^{-1}(B_T)\subset \lfloor B\rfloor$.
Note that this transformation does not affect the coregularity of $(X,B)$ nor such of the general fiber of $f$. 
Passing to a dlt modification, we may assume that $(X,B)$ is dlt.
Let $D_T$ be a prime torus invariant divisor of $T$
and let $S$ be an irreducible component of $\lfloor B\rfloor$ that maps to $D_T$.
Let $B_{D_T}$ be the reduced sum of the torus invariant divisors of $D_T$.
Let $(S,B_S)$ be the Calabi--Yau pair obtained by adjunction of $(Y,B_Y)$ to $S$.
We have an induced projective morphism $S\rightarrow D_T$. 
Let $p_S\colon S \rightarrow D'_T$ and $q\colon D'_T \rightarrow D_T$ be the induced Stein factorization.
By the canonical bundle formula applied to $(S,B_S)$ with respect to $p_S$
we obtain a pair $(D'_T,B_{D'_T})$ for which $q^*(K_{D_T}+B_{D_T})=K_{D'_T}+B_{D'_T}$.
In particular, we know that $D'_T\rightarrow D_T$ is unramified over the torus.
By~\cite[Proposition 3.32]{Mor20c}, we conclude that $(D'_T,B_{D'_T})$ is a toric Calabi--Yau pair. 
Hence, we have a crepant fibration $p_S\colon (S,B_S)\rightarrow (D'_T,B_{D'_T})$ to a toric Calabi--Yau pair of smaller dimension than such of $T$. 
Furthermore, by construction, we have $p_S^{-1}(B_{D'_t})\subset \lfloor B_S\rfloor$.
By induction on the dimension of the base of the fibration, we conclude that  
\begin{equation}\label{eq5}
{\rm coreg}(S_t,B_{S_t})={\rm coreg}(S,B_S)={\rm coreg}(Y,B_Y)={\rm coreg}(X,B)
\end{equation}
holds for $t\in D'_T$ general.
Since $D'_T\rightarrow D_T$ is a finite \'etale morphism on the algebraic torus, the same 
equalities as in~\eqref{eq5} hold for general fibers of $(S,B_S)$ over $D_T$.

Let $\psi\colon T'\rightarrow T$ be a projective birational toric morphism 
such that $T'$ admits a $\pp^1$-fibration for which $D_{T'}:=\psi^{-1}_*(D_T)$
is a section. 
We can choose $\psi$ such that it is an isomorphism at the generic point of $D_T$.
Let $r\colon Y\rightarrow X$ be a projective birational morphism from a normal variety
such that there is a morphism $Y\rightarrow T'$ making the triangular diagram commutative.
We can choose $Y\rightarrow X$ such that it is an isomorphism over the generic point of $D_T$.
Let $(Y,B_Y)$ be the log pull-back of $(X,B)$ to $Y$
and let $\Delta_Y$ be the sum of the components of $B_Y$ which are exceptional over $X$.
Let $\Gamma_Y:={\rm Ex}(r)_{\rm red} -\Delta_Y$.
By possibly replacing $(Y,B_Y+\Gamma_Y)$ with a higher dlt model, we may assume 
that $(Y,B_Y+\Gamma_Y)$ is a dlt pair. 
We run a $(K_Y+B_Y+\Gamma_Y)$-MMP over $T'$.
This MMP terminates in a model $X'$ after contracting all the components of $\Gamma_Y$ (see, e.g.,~\cite[Proposition 5.3]{Lai11}).
Let $X'$ be the resulting model and $B'$ be the push-forward of $B_Y$ to $X'$.
Then, $B'$ is an effective divisor, the map $(X',B')\dashrightarrow (X,B)$ is a crepant birational
map that only extracts log canonical places of $(X,B)$.
Furthermore, we have a crepant fibration $\pi'\colon (X',B')\rightarrow (T',B_{T'})$.
As usual, the divisor $B_{T'}$ is the reduced torus invariant boundary of $T'$.
Note that ${\pi'}^{-1}(B_{T'})\subset \lfloor B'\rfloor$.
Further, we have that $X'\dashrightarrow X$ is an isomorphism over
the generic point of $D_T$.
Therefore, we have a commutative diagram of
crepant maps and morphisms:
\[
\xymatrix{
(X,B)\ar[d]_-{\pi} & (X',B')\ar[d]^-{\pi'} \ar@{-->}[l]_-{\psi_X} \\
(T,B_T) & (T',B_{T'})\ar[l]^-{\psi}\ar[d]^-{\pi_0} \\ 
&  (T_0,B_{T_0}) 
}
\]
where $\pi_0$ is a $\pp^1$-fibration
and $\pi_0$ is a toric fibration onto the toric Calabi--Yau pair $(T_0,B_{T_0})$.
Let $t \in T'$ be a general point.
Let $F$ be the fiber of $\pi_0\circ \pi'$ over $\pi_0(t)$.
Then, the variety $F$ is normal and we can write
\[
K_F+B_F=(K_{X'}+B')|_F
\]
where $(F,B_F)$ is a Calabi--Yau pair. 
Furthermore, we have an induced crepant fibration
\begin{equation}\label{eq:crepant-fib-F}
    \pi_F:=\pi'|_{F}\colon 
(F,B_F) \rightarrow (\pp^1;\{0\}+\{\infty\}) 
\end{equation}
for which $\pi_F^{-1}(0)\subset \lfloor B_F\rfloor$
and $\pi_F^{-1}(\infty)\subset \lfloor B_F\rfloor$. 
By construction, we have an isomorphism
\[
\psi_F\colon \pi_F^{-1}(0) \rightarrow \pi^{-1}(t_0)
\]
for some $t_0\in D_T$ general.
The intersection $\pi^{-1}(t_0)\cap S \simeq \pi_S^{-1}(t_0)$ has an irreducible component $V_0$ of dimension $d$. Let $W_0$ be the normalization of such a component.
Then, by~\eqref{eq5} there exists a boundary $B_{W_0}$ on $W_0$ for which ${\rm coreg}(W_0,B_{W_0})={\rm coreg}(X,B)$.
On the other hand, there is an irreducible component $V$ of $\pi_F^{-1}(0)$, of dimension $d$, 
such that $\psi_F(V)\cap S=V_0$.
If $W\rightarrow V$ is the normalization and $(W,B_W)$ is the pair obtained by adjunction of $(F,B_F)$ to $W$, then $(W,B_W)$ and $(W_0,B_{W_0})$ are crepant birational equivalent. Therefore, we get
\begin{equation}\label{eq6}
{\rm coreg}(W,B_W)={\rm coreg}(W_0,B_{W_0})={\rm coreg}(X,B).
\end{equation}
Now, for $t\in \pp^1$ general, we have 
\[
(F_t,B_{F_t})\simeq (X_s,B_s)
\]
where $s\in T$ is a general point. 
By Proposition~\ref{prop:coreg-over-p1} applied to the crepant fibration~\eqref{eq:crepant-fib-F}, we conclude that
\[
{\rm coreg}(X_s,B_s) = {\rm coreg}(F_t,B_{F_t}) =
{\rm coreg}(W,B_W) = {\rm coreg}(X,B).
\]
This finishes the proof of the theorem. 
\end{proof}

\section{Surface fibrations over toric varieties}

In this section, we study the cluster type property for surface fibrations over toric varieties.

\begin{proposition}\label{prop:toric-blow-up-criterion}
Let $(X,B)$ be a Calabi--Yau pair of index one and complexity two.
Let $\pi\colon (X,B)\rightarrow (T,B_T)$ be a crepant fibration to a toric Calabi--Yau pair.
Assume that $\rho(X/T)=2$ and that $(X,B)$ is of cluster type over $T$. Let $(F,B_F)$ be the log general fiber 
of $\pi$. 
Then, there exists:
\begin{enumerate} 
\item a toric blow-up $(\widetilde{F},B_{\widetilde{F}}) \rightarrow (F,B_F)$, 
\item a non-trivial effective divisor $E$ on $\widetilde{F}$ supported on the strict transform of $B_F$ with non-negative self intersection, and 
\item a node of $B_{\widetilde{F}}$ which is not contained in the support of $E$.
\end{enumerate}
\end{proposition}

\begin{proof}
First, we give a structural description of $(F,B_F)$ as a modification of a toric Calabi--Yau surface.
Since $(X,B)$ is cluster type over $T$, there exists a diagram 
\begin{equation}
    \xymatrix{
         & (Y,B_Y) \ar[ld]_{\phi} \ar@{-->}[rd]^{\psi} &  \\
        (X,B) \ar[d]_{\pi} & & (T_0,B_{T_0}) \ar[dll]^{\pi_0} \\
        (T,B_T)
    }
    \label{diagram_clustoverT}
\end{equation}
where $\pi_0$ is a toric fibration, $\phi$ is a dlt modification, 
and $\psi$ is a birational contraction that contracts components of $B_Y$, and precisely two prime divisors with coefficient 0.
The fact that $\psi$ contracts two divisors with coefficient $0$ is because $(X,B)$ has complexity two.
By passing to a higher birational toric model of $T_0$, we may assume that $\psi$ only contracts divisors with coefficient 0.

Restricting to general fibers over $T$ in the diagram~\eqref{diagram_clustoverT}, we get
\begin{equation}
    \xymatrix{
                & (Y_F,B_F) \ar[rd]^{\psi_F} \ar[ld]_{\phi_F} & \\
        (F,B_F) & & (F_0,B_{F_0})
    } 
    \label{diagram_clustFiber}
\end{equation}
where $F_0$ is a toric Calabi--Yau surface.
The morphism $\psi_F$ is a sequence of blow-ups of points of $B_{F_0}$ which are not nodes, 
and all those points lie in at most two prime components of $B_{F_0}$.
The last statement follows as $\psi_F$ is the restriction of $\psi$, 
which is the divisorial contraction of two prime divisors $E_1$ and $E_2$.
The mentioned components of $B_{F_0}$ are the restrictions to $F_0$ of the components of $B_{T_0}$
containing the centers of $E_1$ and $E_2$.
Since $B\to B|_F$ is a bijection in components, and $c(X,B) = 2$, we conclude that
$B_F$ has two components, say $C_1$ and $C_2$.
Thus, the diagram~\eqref{diagram_clustFiber} looks as in Figure~\ref{figure_fiberDiagram}.

    \begin{figure}[htb]
    \centering
\psscalebox{0.6 0.6} 
{
\begin{pspicture}(0,-4.379554)(20.10851,6.5277805)
\psline[linecolor=black, linewidth=0.04, arrowsize=0.013cm 5.0,arrowlength=2.0,arrowinset=0.6]{<-}(5.866201,0.59982085)(8.5662,2.799821)
\psline[linecolor=black, linewidth=0.04, arrowsize=0.013cm 5.0,arrowlength=2.0,arrowinset=0.6]{->}(13.450816,2.3075132)(15.850816,-0.09248685)
\psline[linecolor=black, linewidth=0.04](9.032867,2.5155718)(9.632868,4.715572)
\psline[linecolor=black, linewidth=0.04](9.232867,4.215572)(11.432867,5.115572)
\psline[linecolor=black, linewidth=0.04](11.032867,5.115572)(13.032867,4.215572)
\psline[linecolor=black, linewidth=0.04](12.832868,4.4155717)(12.532867,2.4155717)
\psline[linecolor=black, linewidth=0.04](10.632868,1.8155718)(12.832868,2.8155718)
\psline[linecolor=black, linewidth=0.04](8.932867,2.9155717)(11.132868,1.9155718)
\psbezier[linecolor=red, linewidth=0.04](11.432867,4.4155717)(11.932867,4.815572)(12.132868,5.615572)(12.232867,6.51557175409221)
\psbezier[linecolor=red, linewidth=0.04](11.868161,4.3332186)(12.368161,4.7332187)(12.568162,5.533219)(12.668161,6.433218812915741)
\psbezier[linecolor=red, linewidth=0.04](12.185808,4.133219)(12.685808,4.533219)(12.885809,5.3332186)(12.985808,6.233218812915741)
\psbezier[linecolor=red, linewidth=0.04](9.550514,3.050866)(9.650515,2.750866)(9.150515,2.0332189)(8.450515,1.7332188129157522)
\psbezier[linecolor=red, linewidth=0.04](9.868161,2.8508658)(9.968162,2.550866)(9.468162,1.8332188)(8.768162,1.53321881291576)
\psbezier[linecolor=red, linewidth=0.04](10.185808,2.768513)(10.285809,2.468513)(9.785809,1.7508658)(9.085809,1.4508658717392917)
\psbezier[linecolor=red, linewidth=0.04](10.385809,2.568513)(10.485808,2.268513)(9.985808,1.5508659)(9.285809,1.2508658717392813)
\psline[linecolor=black, linewidth=0.04](15.15142,-2.9513965)(15.751419,-0.7513966)
\psline[linecolor=black, linewidth=0.04](15.351419,-1.2513965)(17.55142,-0.35139656)
\psline[linecolor=black, linewidth=0.04](17.151419,-0.35139656)(19.151419,-1.2513965)
\psline[linecolor=black, linewidth=0.04](18.95142,-1.0513966)(18.651419,-3.0513966)
\psline[linecolor=black, linewidth=0.04](16.75142,-3.6513965)(18.95142,-2.6513965)
\psline[linecolor=black, linewidth=0.04](15.051419,-2.5513966)(17.25142,-3.5513966)
\psbezier[linecolor=red, linewidth=0.04](4.868661,-0.6208435)(5.5351243,-0.59436524)(5.847214,-0.12983486)(6.6051307,-0.05524982495244444)
\psbezier[linecolor=red, linewidth=0.04](4.8605914,-0.6187868)(5.3806405,-0.356054)(5.763149,0.42994633)(6.5210648,0.5045313717930918)
\psbezier[linecolor=red, linewidth=0.04](4.8804836,-0.6099651)(5.455214,-0.42963406)(5.6053576,0.093818046)(6.3632736,0.1684030771082382)
\psbezier[linecolor=red, linewidth=0.04](1.9372648,-2.9086607)(1.1598068,-2.9832206)(0.7232178,-2.7630098)(0.01751963,-3.0493489977601063)
\psbezier[linecolor=red, linewidth=0.04](1.9879223,-2.8961174)(1.2104644,-2.9706771)(0.9823866,-3.4217014)(0.27668843,-3.7080405886628443)
\psbezier[linecolor=red, linewidth=0.04](1.9673065,-2.887153)(1.4672536,-3.2229824)(1.4891566,-3.762649)(0.78345853,-4.04898827211287)
\psbezier[linecolor=red, linewidth=0.04](1.9315457,-2.8935397)(1.6064758,-3.2867362)(1.9228048,-3.9767966)(1.4255297,-4.3536133681622395)
\psbezier[linecolor=black, linewidth=0.04](5.6543236,-0.8342995)(3.0580711,0.22878632)(1.6661774,-2.5434892)(1.8949695,-4.011149912445278)
\psbezier[linecolor=black, linewidth=0.04](1.0235937,-2.6172075)(3.6336308,-3.747233)(5.13744,-1.0737511)(4.9504037,0.3739477321128106)
\rput[bl](13.543124,3.5921285){$(Y_F,B_F)$}
\rput[bl](18.658508,-3.846333){$(F_0,B_{F_0})$}
\rput[bl](4.3585086,-3.146333){$(F,B_F)$}
\rput[bl](6.7123547,2.1459746){$\phi_F$}
\rput[bl](14.858508,1.5075132){$\psi_F$}
\end{pspicture}
}
\caption{Example of Diagram~\ref{diagram_clustFiber}.}
\label{figure_fiberDiagram}
\end{figure}

We now divide the proof into two cases depending on the 
two components $D_1$ and $D_2$ of $B_{F_0}$ that we blow up in $\psi_F$.\\

\textit{Case 1.} The divisors $D_1$ and $D_2$ correspond to rays $\rho$ and $-\rho$ in the fan of $F_0$.\\ 

As $D_1$ and $D_2$ correspond to $\rho$ and $-\rho$, this gives us a toric fibration $f\colon F_0 \to \pp^1$

    \begin{figure}[htb]
    \centering
\psscalebox{0.6 0.6} 
{
\begin{pspicture}(0,-4.4351945)(14.381579,4.2944174)
\psline[linecolor=black, linewidth=0.04](2.2684212,2.864806)(5.068421,2.864806)
\psline[linecolor=black, linewidth=0.04](2.668421,3.064806)(2.0684211,1.5648059)
\psline[linecolor=black, linewidth=0.04](2.168421,2.364806)(2.668421,0.86480594)
\psline[linecolor=black, linewidth=0.04](2.668421,1.5648059)(2.025564,-0.020908378)(2.025564,-0.020908378)
\psline[linecolor=black, linewidth=0.04](2.0684211,0.6648059)(2.5684211,-0.8351941)
\psline[linecolor=black, linewidth=0.04](4.571565,-0.8351941)(5.171565,0.72017014)
\psline[linecolor=black, linewidth=0.04](5.071565,-0.07982987)(4.571565,1.4201702)
\psline[linecolor=black, linewidth=0.04](4.571565,0.72017014)(5.214422,2.3058844)(5.214422,2.3058844)
\psline[linecolor=black, linewidth=0.04](5.171565,1.6201701)(4.671565,3.064806)
\psline[linecolor=black, linewidth=0.04](2.2684212,-0.63519406)(5.068421,-0.63519406)
\psline[linecolor=black, linewidth=0.04](9.083806,2.8186522)(11.883805,2.8186522)
\psline[linecolor=black, linewidth=0.04](9.483806,3.018652)(8.883805,1.5186521)
\psline[linecolor=black, linewidth=0.04](8.983806,2.3186522)(9.483806,0.81865203)
\psline[linecolor=black, linewidth=0.04](9.483806,1.5186521)(8.840948,-0.06706222)(8.840948,-0.06706222)
\psline[linecolor=black, linewidth=0.04](8.883805,0.61865205)(9.383805,-0.88134795)
\psline[linecolor=black, linewidth=0.04](11.38695,-0.88134795)(11.98695,0.6740163)
\psline[linecolor=black, linewidth=0.04](11.88695,-0.12598372)(11.38695,1.3740163)
\psline[linecolor=black, linewidth=0.04](11.38695,0.6740163)(12.029807,2.2597306)(12.029807,2.2597306)
\psline[linecolor=black, linewidth=0.04](11.98695,1.5740163)(11.48695,3.018652)
\psline[linecolor=black, linewidth=0.04](9.083806,-0.68134797)(11.883805,-0.68134797)
\psline[linecolor=black, linewidth=0.04, arrowsize=0.013cm 5.0,arrowlength=2.0,arrowinset=0.6]{->}(5.768421,1.1648059)(8.368421,1.1648059)
\psline[linecolor=black, linewidth=0.04, arrowsize=0.013cm 5.0,arrowlength=2.0,arrowinset=0.6]{->}(10.368421,-1.4351941)(10.368421,-3.435194)
\psline[linecolor=black, linewidth=0.04](8.606883,-3.935194)(12.206882,-3.935194)
\psbezier[linecolor=red, linewidth=0.04](2.968421,2.4648058)(3.2684212,2.9648058)(2.7684212,4.164806)(3.2684212,4.264805908203125)
\psbezier[linecolor=red, linewidth=0.04](4.128421,2.344806)(4.428421,2.844806)(3.928421,4.044806)(4.428421,4.144805908203125)
\psbezier[linecolor=red, linewidth=0.04](3.5284212,2.384806)(3.828421,2.884806)(3.328421,4.084806)(3.828421,4.184805908203125)
\psbezier[linecolor=red, linewidth=0.04](2.868421,-0.1351941)(2.868421,-0.9351941)(3.368421,-1.2351941)(3.0684211,-1.735194091796875)
\psbezier[linecolor=red, linewidth=0.04](3.148421,-0.17519408)(3.148421,-0.9751941)(3.648421,-1.275194)(3.348421,-1.775194091796875)
\psbezier[linecolor=red, linewidth=0.04](3.428421,-0.21519409)(3.428421,-1.015194)(3.928421,-1.3151941)(3.628421,-1.815194091796875)
\psbezier[linecolor=red, linewidth=0.04](3.7884212,-0.2551941)(3.7884212,-1.0551941)(4.288421,-1.3551941)(3.988421,-1.855194091796875)
\rput[bl](10.247369,2.933227){$D_1$}
\rput[bl](10.184211,-1.1088783){$D_2$}
\rput[bl](10.636842,-2.335194){$f$}
\psdots[linecolor=black, dotsize=0.2](9.168421,-3.935194)
\psdots[linecolor=black, dotsize=0.2](11.668421,-3.935194)
\rput[bl](9.078947,-4.429931){$0$}
\rput[bl](11.473684,-4.435194){$\infty$}
\rput[bl](12.621053,-4.0246677){$\pp^1$}
\rput[bl](12.931579,0.9069112){$(F_0,B_{F_0})$}
\rput[bl](0.0,0.8279638){$(Y_F,B_{Y_F})$}
\end{pspicture}
}
\caption{Case 1.}
\label{figure_Case1}
\end{figure}

By construction, both $f^{-1}(0)$ and $f^{-1}(\infty)$ support nef effective divisors, 
so they cannot be contracted completely.
Thus, $f^{-1}(0)$ and $f^{-1}(\infty)$ contains the strict transform of $C_1$ and $C_2$ respectively. 
Let $\tilde{C}_1$ and $\tilde{C}_2$ the strict transform of $C_1$ and $C_2$. 
Then we can contract all components of $B_{Y_F}$ which are not $D_1$, $D_2$, $\tilde{C}_1$ and $\tilde{C}_2$. 
See Figure~\ref{figure_Case1contracted}.
In this model, $\tilde{C}_1^2 = \tilde{C}_2^2 = 0$, and $B_{\tilde{F}}$ has a node outside 
both $\tilde{C}_1$ and $\tilde{C}_2$.

    \begin{figure}[htb]
    \centering
        \psscalebox{0.6 0.6} 
{
\begin{pspicture}(0,-5.465865)(12.6335535,5.3250885)
\psline[linecolor=black, linewidth=0.04](2.2684212,3.8954766)(5.068421,3.8954766)
\psline[linecolor=black, linewidth=0.04](2.668421,4.0954766)(2.0684211,2.5954766)
\psline[linecolor=black, linewidth=0.04](2.668421,2.5954766)(2.025564,1.0097624)(2.025564,1.0097624)
\psline[linecolor=black, linewidth=0.04](2.0684211,1.6954767)(2.5684211,0.19547668)
\psline[linecolor=black, linewidth=0.04](4.571565,0.19547668)(5.171565,1.7508409)
\psline[linecolor=black, linewidth=0.04](5.071565,0.9508409)(4.571565,2.450841)
\psline[linecolor=black, linewidth=0.04](5.171565,2.650841)(4.671565,4.0954766)
\psline[linecolor=black, linewidth=0.04](2.2684212,0.3954767)(5.068421,0.3954767)
\psbezier[linecolor=red, linewidth=0.04](2.968421,3.4954767)(3.2684212,3.9954767)(2.7684212,5.1954765)(3.2684212,5.295476684570312)
\psbezier[linecolor=red, linewidth=0.04](4.128421,3.3754766)(4.428421,3.8754766)(3.928421,5.0754766)(4.428421,5.175476684570312)
\psbezier[linecolor=red, linewidth=0.04](3.5284212,3.4154768)(3.828421,3.9154768)(3.328421,5.1154766)(3.828421,5.215476684570312)
\psbezier[linecolor=red, linewidth=0.04](2.868421,0.8954767)(2.868421,0.09547669)(3.368421,-0.20452331)(3.0684211,-0.7045233154296875)
\psbezier[linecolor=red, linewidth=0.04](3.148421,0.8554767)(3.148421,0.055476684)(3.648421,-0.24452332)(3.348421,-0.7445233154296875)
\psbezier[linecolor=red, linewidth=0.04](3.428421,0.81547666)(3.428421,0.015476685)(3.928421,-0.2845233)(3.628421,-0.7845233154296875)
\psbezier[linecolor=red, linewidth=0.04](3.7884212,0.7754767)(3.7884212,-0.024523316)(4.288421,-0.32452333)(3.988421,-0.8245233154296875)
\rput[bl](0.0,1.8586346){$(Y_F,B_{Y_F})$}
\psbezier[linecolor=black, linewidth=0.04](10.905668,3.7530982)(12.290283,2.922329)(12.321053,0.3069443)(11.028745,-0.40074801178596997)
\psbezier[linecolor=black, linewidth=0.04](11.98601,3.6231837)(10.109087,2.7001066)(9.767206,0.43002123)(12.259514,-0.1853633964013568)
\rput[bl](1.6376518,3.273746){$\tilde{C}_1$}
\rput[bl](5.3333783,3.2498143){$\tilde{C}_2$}
\psline[linecolor=blue, linewidth=0.04](4.571565,1.7508409)(5.214422,3.3365552)(5.214422,3.3365552)
\psline[linecolor=blue, linewidth=0.04](2.168421,3.3954766)(2.668421,1.8954767)
\psline[linecolor=black, linewidth=0.04, arrowsize=0.013cm 5.0,arrowlength=2.0,arrowinset=0.6]{->}(6.368421,1.9954767)(9.668421,1.9954767)
\rput[bl](10.296053,2.8027136){$C_1$}
\rput[bl](12.183553,1.6777135){$C_2$}
\psline[linecolor=black, linewidth=0.04](6.3631577,-2.1992602)(8.868421,-2.2045233)
\psline[linecolor=black, linewidth=0.04](6.368421,-4.204523)(8.868421,-4.204523)
\psbezier[linecolor=red, linewidth=0.04](7.273684,-2.59926)(7.573684,-2.09926)(7.073684,-0.89926016)(7.573684,-0.7992601575349488)
\psbezier[linecolor=red, linewidth=0.04](7.9073687,-2.613997)(8.207369,-2.113997)(7.7073684,-0.913997)(8.207369,-0.8139969996402112)
\psbezier[linecolor=red, linewidth=0.04](7.623158,-2.6792603)(7.9231577,-2.1792603)(7.4231577,-0.97926015)(7.9231577,-0.8792601575349488)
\psbezier[linecolor=red, linewidth=0.04](6.963158,-3.725576)(6.963158,-4.525576)(7.463158,-4.825576)(7.163158,-5.325575947008636)
\psbezier[linecolor=red, linewidth=0.04](7.243158,-3.765576)(7.243158,-4.565576)(7.743158,-4.865576)(7.4431577,-5.365575947008636)
\psbezier[linecolor=red, linewidth=0.04](7.523158,-3.8055758)(7.523158,-4.605576)(8.023158,-4.9055758)(7.723158,-5.405575947008636)
\psbezier[linecolor=red, linewidth=0.04](7.8831577,-3.845576)(7.8831577,-4.645576)(8.383158,-4.9455757)(8.083158,-5.445575947008636)
\psline[linecolor=blue, linewidth=0.04](6.568421,-1.9045234)(6.568421,-4.604523)
\psline[linecolor=blue, linewidth=0.04](8.568421,-1.9045234)(8.568421,-4.5045233)
\psline[linecolor=black, linewidth=0.04, arrowsize=0.013cm 5.0,arrowlength=2.0,arrowinset=0.6]{->}(3.968421,-1.2045233)(5.968421,-3.0045233)
\psline[linecolor=black, linewidth=0.04, arrowsize=0.013cm 5.0,arrowlength=2.0,arrowinset=0.6]{->}(9.168421,-3.0045233)(11.168421,-1.4045234)
\rput[bl](13.252631,1.7638978){$(F,B_F)$}
\rput[bl](9.457894,-3.7097864){$(\tilde{F},B_{\tilde{F}})$}
\end{pspicture}
}   
    \caption{Contraction of Case 1. In blue, the components corresponding to the strict transform of $C_1$ and $C_2$}
    \label{figure_Case1contracted}
    \end{figure}
\mbox{}

\textit{Case 2.} The rays corresponding to $D_1$ and $D_2$ in the fan of $F_0$ are not in the same line.\\ 

Passing to a higher toric model of $F_0$, 
we may assume that there is a toric fibration $f\colon F_0 \to \pp^1$ 
such that $D_1$ and $D_2$ are supported on $f^{-1}(\infty)$. 
Indeed, let $\rho_1$ and $\rho_2$ be the generators of the rays of the fan of $F_0$ corresponding to $D_1$ and $D_2$, choose a linear form $L$ on $\qq^2$ such that $L(\rho_1)>0$ and $L(\rho_2)>0$. 
Then, we perform the toric blow-up on $F_0$ which corresponds to adding the ray given by $L^\perp$. Thus, we obtain a diagram as in Figure~\ref{figure_Case2}.

    \begin{figure}[htb]
    \centering
    \psscalebox{0.6 0.6} 
{
\begin{pspicture}(0,-5.041286)(14.381579,2.476142)
\psline[linecolor=black, linewidth=0.04](2.2684212,2.258714)(5.068421,2.258714)
\psline[linecolor=black, linewidth=0.04](2.668421,2.458714)(2.0684211,0.958714)
\psline[linecolor=black, linewidth=0.04](2.168421,1.758714)(2.668421,0.258714)
\psline[linecolor=black, linewidth=0.04](2.668421,0.958714)(2.025564,-0.6270003)(2.025564,-0.6270003)
\psline[linecolor=black, linewidth=0.04](2.0684211,0.058713987)(2.5684211,-1.441286)
\psline[linecolor=black, linewidth=0.04](4.571565,-1.441286)(5.171565,0.11407821)
\psline[linecolor=black, linewidth=0.04](5.071565,-0.6859218)(4.571565,0.8140782)
\psline[linecolor=black, linewidth=0.04](4.571565,0.11407821)(5.214422,1.6997925)(5.214422,1.6997925)
\psline[linecolor=black, linewidth=0.04](5.171565,1.0140783)(4.671565,2.458714)
\psline[linecolor=black, linewidth=0.04](2.2684212,-1.241286)(5.068421,-1.241286)
\psline[linecolor=black, linewidth=0.04](9.083806,2.2125602)(11.883805,2.2125602)
\psline[linecolor=black, linewidth=0.04](9.483806,2.4125602)(8.883805,0.91256016)
\psline[linecolor=black, linewidth=0.04](8.983806,1.7125602)(9.483806,0.21256015)
\psline[linecolor=black, linewidth=0.04](9.483806,0.91256016)(8.840948,-0.6731541)(8.840948,-0.6731541)
\psline[linecolor=black, linewidth=0.04](8.883805,0.012560143)(9.383805,-1.4874399)
\psline[linecolor=black, linewidth=0.04](11.38695,-1.4874399)(11.98695,0.067924365)
\psline[linecolor=black, linewidth=0.04](11.88695,-0.73207563)(11.38695,0.76792437)
\psline[linecolor=black, linewidth=0.04](11.38695,0.067924365)(12.029807,1.6536386)(12.029807,1.6536386)
\psline[linecolor=black, linewidth=0.04](11.98695,0.96792436)(11.48695,2.4125602)
\psline[linecolor=black, linewidth=0.04](9.083806,-1.2874398)(11.883805,-1.2874398)
\psline[linecolor=black, linewidth=0.04, arrowsize=0.013cm 5.0,arrowlength=2.0,arrowinset=0.6]{->}(5.768421,0.558714)(8.368421,0.558714)
\psline[linecolor=black, linewidth=0.04, arrowsize=0.013cm 5.0,arrowlength=2.0,arrowinset=0.6]{->}(10.368421,-2.041286)(10.368421,-4.041286)
\psline[linecolor=black, linewidth=0.04](8.606883,-4.541286)(12.206882,-4.541286)
\rput[bl](11.170445,1.5579042){$D_1$}
\rput[bl](12.030364,-0.17650868){$D_2$}
\rput[bl](10.636842,-2.941286){$f$}
\psdots[linecolor=black, dotsize=0.2](9.168421,-4.541286)
\psdots[linecolor=black, dotsize=0.2](11.668421,-4.541286)
\rput[bl](9.078947,-5.0360227){$0$}
\rput[bl](11.473684,-5.041286){$\infty$}
\rput[bl](12.621053,-4.6307597){$\pp^1$}
\rput[bl](12.931579,0.30081925){$(F_0,B_{F_0})$}
\rput[bl](0.0,0.22187188){$(Y_F,B_{Y_F})$}
\psbezier[linecolor=red, linewidth=0.04](4.542105,-0.9360229)(4.9421053,-0.83602285)(5.368421,-1.441286)(5.568421,-1.3412860107421876)
\psbezier[linecolor=red, linewidth=0.04](4.6578946,2.0008192)(5.0578947,2.1008193)(5.568421,2.058714)(5.768421,2.1587139892578127)
\psbezier[linecolor=red, linewidth=0.04](4.6473684,1.9060824)(5.0473685,2.0060823)(5.5578947,1.9639771)(5.7578945,2.0639771471525497)
\psbezier[linecolor=red, linewidth=0.04](4.6368423,1.8113456)(5.036842,1.9113456)(5.5473685,1.8692403)(5.7473683,1.9692403050472864)
\psbezier[linecolor=red, linewidth=0.04](4.731579,1.7166088)(5.131579,1.8166087)(5.642105,1.7745035)(5.8421054,1.8745034629420234)
\psbezier[linecolor=red, linewidth=0.04](4.6368423,-0.8202334)(5.036842,-0.7202334)(5.463158,-1.3254966)(5.663158,-1.2254965370579771)
\psbezier[linecolor=red, linewidth=0.04](4.6263156,-0.70444393)(5.0263157,-0.6044439)(5.4526315,-1.209707)(5.6526318,-1.1097070633737673)
\rput[bl](3.468421,2.483421){$C_0$}
\rput[bl](3.3473685,-1.6744737){$C_{\infty}$}
\end{pspicture}
}
    \caption{Case 2.}
    \label{figure_Case2}
    \end{figure}

As argued above, $\tilde{C}_1$ must lie in $f^{-1}(0)$, while $\tilde{C}_2$ could lie anywhere in $B_{Y_F}$.
We contract all components of $B_{Y_F}$ that are not in $\{\tilde{C}_1,\tilde{C}_2,C_0,C_{\infty}\}$.
Let $f^{-1}(0) = E$, which is contained in the strict transform of $B_F$. 
Then we have $E^2 = 0$, and there is a node outside of $E$.
In Figure~\ref{figure_Case2cases}, we list the possible combinatorial cases.

    \begin{figure}[hbt]
    \centering
        \hfill\begin{minipage}{.45\textwidth}
            \psscalebox{0.6 0.6} 
{
\begin{pspicture}(0,-3.2383022)(13.651969,1.256197)
\psbezier[linecolor=red, linewidth=0.04](12.880771,-2.6954577)(13.217372,-2.838268)(13.262064,-2.983099)(13.489188,-3.054340722467432)
\psbezier[linecolor=red, linewidth=0.04](12.895253,-2.6685398)(13.0129,-2.9999125)(13.264254,-3.1611333)(13.370135,-3.210152952892681)
\psbezier[linecolor=red, linewidth=0.04](12.869659,-2.6565688)(13.128483,-2.7604904)(13.472748,-2.8010476)(13.55329,-2.8910928592195795)
\psbezier[linecolor=red, linewidth=0.04](12.880186,0.27573258)(13.076264,0.6345561)(12.959339,0.8591175)(13.15934,0.9591175083282434)
\psbezier[linecolor=red, linewidth=0.04](12.84613,0.24374083)(13.199071,0.4849173)(13.113519,0.90555716)(13.337049,0.895753215242587)
\psbezier[linecolor=red, linewidth=0.04](12.85129,0.23543587)(13.25129,0.34013528)(13.134365,0.6884738)(13.514757,0.8010163731373285)
\psbezier[linecolor=red, linewidth=0.04](12.867596,0.24250244)(13.267595,0.34250244)(13.433023,0.55922073)(13.633023,0.659220707502651)
\psline[linecolor=black, linewidth=0.04](2.2684212,0.7472496)(5.068421,0.7472496)
\psline[linecolor=black, linewidth=0.04](2.668421,0.9472496)(2.0684211,-0.5527504)
\psline[linecolor=blue, linewidth=0.04](2.168421,0.2472496)(2.668421,-1.2527504)
\psline[linecolor=black, linewidth=0.04](2.668421,-0.5527504)(2.025564,-2.1384647)(2.025564,-2.1384647)
\psline[linecolor=black, linewidth=0.04](2.0684211,-1.4527504)(2.5684211,-2.9527504)
\psline[linecolor=black, linewidth=0.04](4.571565,-2.9527504)(5.171565,-1.3973862)
\psline[linecolor=black, linewidth=0.04](5.071565,-2.1973863)(4.571565,-0.6973862)
\psline[linecolor=blue, linewidth=0.04](4.571565,-1.3973862)(5.214422,0.18832812)(5.214422,0.18832812)
\psline[linecolor=black, linewidth=0.04](5.171565,-0.4973862)(4.671565,0.9472496)
\psline[linecolor=black, linewidth=0.04](2.2684212,-2.7527504)(5.068421,-2.7527504)
\psline[linecolor=black, linewidth=0.04, arrowsize=0.013cm 5.0,arrowlength=2.0,arrowinset=0.6]{->}(5.768421,-0.9527504)(8.368421,-0.9527504)
\rput[bl](0.0,-1.2895925){$(Y_F,B_{Y_F})$}
\psbezier[linecolor=red, linewidth=0.04](4.542105,-2.4474874)(4.9421053,-2.3474872)(5.368421,-2.9527504)(5.568421,-2.8527503967285157)
\psbezier[linecolor=red, linewidth=0.04](4.6578946,0.48935488)(5.0578947,0.5893549)(5.568421,0.5472496)(5.768421,0.6472496032714844)
\psbezier[linecolor=red, linewidth=0.04](4.6473684,0.39461803)(5.0473685,0.49461803)(5.5578947,0.45251277)(5.7578945,0.5525127611662214)
\psbezier[linecolor=red, linewidth=0.04](4.6368423,0.2998812)(5.036842,0.39988118)(5.5473685,0.35777593)(5.7473683,0.4577759190609584)
\psbezier[linecolor=red, linewidth=0.04](4.731579,0.20514435)(5.131579,0.30514434)(5.642105,0.26303908)(5.8421054,0.3630390769556954)
\psbezier[linecolor=red, linewidth=0.04](4.6368423,-2.3316977)(5.036842,-2.2316978)(5.463158,-2.836961)(5.663158,-2.7369609230443053)
\psbezier[linecolor=red, linewidth=0.04](4.6263156,-2.2159083)(5.0263157,-2.1159084)(5.4526315,-2.7211714)(5.6526318,-2.6211714493600953)
\rput[bl](3.468421,0.926197){$C_0$}
\rput[bl](3.3473685,-3.2316978){$C_{\infty}$}
\psline[linecolor=black, linewidth=0.04](9.668421,0.2472496)(13.168421,0.2472496)
\psline[linecolor=black, linewidth=0.04](9.668421,-2.6527505)(13.168421,-2.6527505)
\psline[linecolor=blue, linewidth=0.04](9.968421,0.5472496)(9.968421,-3.0527503)
\psline[linecolor=blue, linewidth=0.04](12.868421,0.5472496)(12.868421,-3.0527503)
\rput[bl](9.416516,-1.279179){$\tilde{C}_1$}
\rput[bl](12.997469,-1.2696551){$\tilde{C}_2$}
\rput[bl](11.235564,0.34939247){$C_0$}
\rput[bl](11.168898,-3.0887027){$C_{\infty}$}
\end{pspicture}
}
        \end{minipage} \hfill \hfill
        \begin{minipage}{.35\textwidth}
            In this case $\tilde{C}_1^2 = 0$ and $\tilde{C}_2$ could be negative. There is a node outside $\tilde{C}_1$.
        \end{minipage}\hfill
        \vskip 11pt
        \hfill \begin{minipage}{0.45\textwidth}
            \psscalebox{0.6 0.6} 
{
\begin{pspicture}(0,-3.235)(20.558445,1.2528946)
\psbezier[linecolor=red, linewidth=0.04](12.832617,0.241957)(13.128768,0.045213956)(13.164472,-0.17602132)(13.432045,-0.15288112483442887)
\psbezier[linecolor=red, linewidth=0.04](12.83811,0.23291972)(12.951262,-0.07148653)(13.20711,-0.25967374)(13.312993,-0.30869335525967867)
\psbezier[linecolor=red, linewidth=0.04](12.8080225,0.26736274)(13.066846,0.16344117)(13.379649,0.024007494)(13.496147,0.01935550245836623)
\psbezier[linecolor=red, linewidth=0.04](12.793966,0.30724302)(12.570223,0.6495)(12.317061,0.65244615)(12.323332,0.8759650214603645)
\psbezier[linecolor=red, linewidth=0.04](12.80609,0.2621178)(12.760114,0.6871117)(12.348914,0.81028587)(12.463089,1.002706415585343)
\psbezier[linecolor=red, linewidth=0.04](12.815847,0.2627456)(12.912398,0.6647903)(12.550123,0.72617376)(12.630503,1.1146362782555568)
\psbezier[linecolor=red, linewidth=0.04](12.817316,0.28045568)(12.91801,0.6802817)(12.805066,0.92842907)(12.811336,1.1519479662884045)
\psline[linecolor=blue, linewidth=0.04](2.2684212,0.74394727)(5.068421,0.74394727)
\psline[linecolor=black, linewidth=0.04](2.668421,0.9439473)(2.0684211,-0.5560527)
\psline[linecolor=blue, linewidth=0.04](2.168421,0.2439473)(2.668421,-1.2560527)
\psline[linecolor=black, linewidth=0.04](2.668421,-0.5560527)(2.025564,-2.141767)(2.025564,-2.141767)
\psline[linecolor=black, linewidth=0.04](2.0684211,-1.4560527)(2.5684211,-2.9560528)
\psline[linecolor=black, linewidth=0.04](4.571565,-2.9560528)(5.171565,-1.4006885)
\psline[linecolor=black, linewidth=0.04](5.071565,-2.2006886)(4.571565,-0.7006885)
\psline[linecolor=black, linewidth=0.04](4.571565,-1.4006885)(5.214422,0.1850258)(5.214422,0.1850258)
\psline[linecolor=black, linewidth=0.04](5.171565,-0.5006885)(4.671565,0.9439473)
\psline[linecolor=black, linewidth=0.04](2.2684212,-2.7560527)(5.068421,-2.7560527)
\psline[linecolor=black, linewidth=0.04, arrowsize=0.013cm 5.0,arrowlength=2.0,arrowinset=0.6]{->}(5.768421,-0.9560527)(8.368421,-0.9560527)
\rput[bl](0.0,-1.2928948){$(Y_F,B_{Y_F})$}
\psbezier[linecolor=red, linewidth=0.04](4.542105,-2.4507895)(4.9421053,-2.3507895)(5.368421,-2.9560528)(5.568421,-2.8560527038574217)
\psbezier[linecolor=red, linewidth=0.04](4.6578946,0.48605257)(5.0578947,0.58605254)(5.568421,0.5439473)(5.768421,0.6439472961425782)
\psbezier[linecolor=red, linewidth=0.04](4.6473684,0.39131573)(5.0473685,0.49131572)(5.5578947,0.44921046)(5.7578945,0.5492104540373152)
\psbezier[linecolor=red, linewidth=0.04](4.6368423,0.29657888)(5.036842,0.39657888)(5.5473685,0.35447362)(5.7473683,0.4544736119320521)
\psbezier[linecolor=red, linewidth=0.04](4.731579,0.20184204)(5.131579,0.30184203)(5.642105,0.25973678)(5.8421054,0.3597367698267891)
\psbezier[linecolor=red, linewidth=0.04](4.6368423,-2.335)(5.036842,-2.2350001)(5.463158,-2.8402631)(5.663158,-2.7402632301732113)
\psbezier[linecolor=red, linewidth=0.04](4.6263156,-2.2192106)(5.0263157,-2.1192105)(5.4526315,-2.7244737)(5.6526318,-2.6244737564890017)
\rput[bl](3.468421,0.92289466){$C_0$}
\rput[bl](3.3473685,-3.2350001){$C_{\infty}$}
\psline[linecolor=blue, linewidth=0.04](9.668421,0.2439473)(13.168421,0.2439473)
\psline[linecolor=black, linewidth=0.04](9.751755,-2.9227195)(13.070802,0.49156633)
\psline[linecolor=blue, linewidth=0.04](9.968421,0.5439473)(9.968421,-3.0560527)
\rput[bl](9.416516,-1.2824813){$\tilde{C}_1$}
\rput[bl](11.547469,-1.5991479){$C_{\infty}$}
\rput[bl](10.691324,0.37031826){$C_0 = \tilde{C}_2$}
\end{pspicture}
}
        \end{minipage} \hfill \hfill
        \begin{minipage}{0.35\textwidth}
            In this case, $\tilde{C}_1^2 =0$ and $\tilde{C}_2^2$ could be negative. There is a node outside $\tilde{C}_1$.
        \end{minipage}
        \vskip 11pt
        \hfill \begin{minipage}{0.45\textwidth}
            \psscalebox{0.6 0.6} 
{
\begin{pspicture}(0,-3.235)(24.768808,1.2528946)
\psbezier[linecolor=red, linewidth=0.04](13.387839,-1.1263092)(13.541886,-1.4467512)(13.459536,-1.6551694)(13.701367,-1.7719959552687987)
\psbezier[linecolor=red, linewidth=0.04](13.387943,-1.1368846)(13.329673,-1.4563704)(13.453446,-1.7488643)(13.519417,-1.8451030096714431)
\psbezier[linecolor=red, linewidth=0.04](13.379679,-1.0919036)(13.549067,-1.3134825)(13.746711,-1.5931695)(13.844476,-1.6566935968138177)
\psbezier[linecolor=red, linewidth=0.04](13.374612,-1.0267965)(13.150867,-0.68453944)(12.897707,-0.68159336)(12.903976,-0.458074458558894)
\psbezier[linecolor=red, linewidth=0.04](13.386735,-1.0719217)(13.340759,-0.64692783)(12.92956,-0.5237536)(13.043735,-0.33133306443391347)
\psbezier[linecolor=red, linewidth=0.04](13.396492,-1.071294)(13.493043,-0.66924924)(13.130768,-0.60786575)(13.211147,-0.2194032017637005)
\psbezier[linecolor=red, linewidth=0.04](13.397961,-1.0535837)(13.498654,-0.6537578)(13.385711,-0.40561038)(13.391981,-0.1820915137308532)
\psline[linecolor=black, linewidth=0.04](2.2684212,0.74394727)(5.068421,0.74394727)
\psline[linecolor=black, linewidth=0.04](2.668421,0.9439473)(2.0684211,-0.5560527)
\psline[linecolor=blue, linewidth=0.04](2.168421,0.2439473)(2.668421,-1.2560527)
\psline[linecolor=black, linewidth=0.04](2.668421,-0.5560527)(2.025564,-2.141767)(2.025564,-2.141767)
\psline[linecolor=blue, linewidth=0.04](2.0684211,-1.4560527)(2.5684211,-2.9560528)
\psline[linecolor=black, linewidth=0.04](4.571565,-2.9560528)(5.171565,-1.4006885)
\psline[linecolor=black, linewidth=0.04](5.071565,-2.2006886)(4.571565,-0.7006885)
\psline[linecolor=black, linewidth=0.04](4.571565,-1.4006885)(5.214422,0.1850258)(5.214422,0.1850258)
\psline[linecolor=black, linewidth=0.04](5.171565,-0.5006885)(4.671565,0.9439473)
\psline[linecolor=black, linewidth=0.04](2.2684212,-2.7560527)(5.068421,-2.7560527)
\psline[linecolor=black, linewidth=0.04, arrowsize=0.013cm 5.0,arrowlength=2.0,arrowinset=0.6]{->}(5.768421,-0.9560527)(8.368421,-0.9560527)
\rput[bl](0.0,-1.2928948){$(Y_F,B_{Y_F})$}
\psbezier[linecolor=red, linewidth=0.04](4.542105,-2.4507895)(4.9421053,-2.3507895)(5.368421,-2.9560528)(5.568421,-2.8560527038574217)
\psbezier[linecolor=red, linewidth=0.04](4.6578946,0.48605257)(5.0578947,0.58605254)(5.568421,0.5439473)(5.768421,0.6439472961425782)
\psbezier[linecolor=red, linewidth=0.04](4.6473684,0.39131573)(5.0473685,0.49131572)(5.5578947,0.44921046)(5.7578945,0.5492104540373152)
\psbezier[linecolor=red, linewidth=0.04](4.6368423,0.29657888)(5.036842,0.39657888)(5.5473685,0.35447362)(5.7473683,0.4544736119320521)
\psbezier[linecolor=red, linewidth=0.04](4.731579,0.20184204)(5.131579,0.30184203)(5.642105,0.25973678)(5.8421054,0.3597367698267891)
\psbezier[linecolor=red, linewidth=0.04](4.6368423,-2.335)(5.036842,-2.2350001)(5.463158,-2.8402631)(5.663158,-2.7402632301732113)
\psbezier[linecolor=red, linewidth=0.04](4.6263156,-2.2192106)(5.0263157,-2.1192105)(5.4526315,-2.7244737)(5.6526318,-2.6244737564890017)
\rput[bl](3.468421,0.92289466){$C_0$}
\rput[bl](3.3473685,-3.2350001){$C_{\infty}$}
\psline[linecolor=blue, linewidth=0.04](8.868421,-1.3560528)(10.329711,0.50201184)
\psline[linecolor=black, linewidth=0.04](9.751755,-2.9227195)(13.74177,-0.9148853)
\psline[linecolor=blue, linewidth=0.04](8.832937,-0.7979882)(9.968421,-3.0560527)
\rput[bl](8.85288,-2.2824812){$\tilde{C}_1$}
\rput[bl](11.729287,-2.4536934){$C_{\infty}$}
\psline[linecolor=black, linewidth=0.04](9.979975,0.42133972)(13.745807,-1.2524928)
\rput[bl](11.522966,-0.15605271){$C_0$}
\rput[bl](9.195694,-0.2560527){$\tilde{C}_2$}
\end{pspicture}
}
        \end{minipage} \hfill \hfill
        \begin{minipage}{0.35\textwidth}
            In this case, $f^{-1}(0) = m_1\tilde{C}_1 + m_2\tilde{C}_2$ is supported on the strict transform of $B_F$, has self-intersection equal to 0, and there is a node outside of it.
        \end{minipage}
    \caption{The three possible subcases in Case 2.}
    \label{figure_Case2cases}
    \end{figure}

\end{proof}

\begin{proof}[Proof of Theorem~\ref{thm:ct-standard-model-Pic-2}]
First assume that (1), (2), and (3) holds.
We know that $|B| = |B_T| + |B_{\text{hor}}|$, as there are no degenerate divisors.
We compute 
\begin{align*}
    2   &= c(X,B) \\
        &= \dim X + \rho(X) - |B|\\
        &= \dim T + 2 + \rho(T) + 2 - |B_T| + |B_{\text{hor}}|\\
        &= 4 - |B_{\text{hor}}|,
\end{align*}
so $B_{\text{hor}}$ has two components.
Property (2) implies that $B_{\text{hor}} \to B_F$ induces a bijection of the components.
Then $B_F$ consists of two components, $C_1$ and $C_2$. 
By (3), we can assume that $C_1$ and $C_2$ intersect.
Figure~\ref{figure_GenFiber} depicts how the general fiber $(F,B_F)$ looks like.

\begin{figure}[htb]
    \centering
    \hfill \begin{minipage}{.35\textwidth}
    \psscalebox{0.6 0.6} 
{
\begin{pspicture}(-0.5,-3.1611795)(5.0503063,1.1337886)
\psbezier[linecolor=black, linewidth=0.04](0.011699419,-2.4418018)(2.0497947,-2.2679923)(3.8986042,-1.9918017)(4.1021757,1.1224839782714844)
\psbezier[linecolor=black, linewidth=0.04](0.6116994,-3.150135)(0.76408035,-0.23584935)(2.8676517,0.4153411)(5.03908,0.5486744544619568)
\rput[bl](1.3616995,-0.10025412){$C_1$}
\rput[bl](3.4188423,-1.9478731){$C_2$}
\rput[bl](-0.5,-0.8812065){$(F,B_F)$}
\end{pspicture}
}
\end{minipage} 
\begin{minipage}{0.35\textwidth}
    with $C_1^2 >0$.
\end{minipage} \hfill
    \caption{General fiber assuming (1) and (2) and (3).}
    \label{figure_GenFiber}
\end{figure}

Now, write $B_{\text{hor}} = B_1 + B_2$, with $B_i|F = C_i$. 
We run a $(K_X + B - \epsilon B_2)$--MMP over $T$,
which induces a $(K_F + B_F - \epsilon C_2)$--MMP on the general fiber.
After finitely many small modifications, which induce the identity on $F$,
we have a divisorial contraction or a Mori Fiber Space.
We study both the divisorial contraction case and the one of a Mori Fiber Space.\\

\textit{Case 1: Divisorial contraction.} As we are assuming the non-existence of degenerate divisors, 
the divisorial contraction induces a divisorial contraction on $F$. 
On the one hand, $C_1$ cannot be contracted on $F$, as $C_1^2 > 0$. 
On the other hand, we cannot contract $C_2$. Indeed, if $C_2$ is contracted, then $C_2^2 < 0$,
but $(K_F + B_F - \epsilon C_2)\cdot C_2 = -\epsilon C_2^2 >0$, which contradicts the MMP.
Therefore, the divisor contracted has coefficient 0, 
and we obtain that $(X,B)$ has complexity 1.\\

\textit{Case 2: Mori Fiber Space.} This induces a fibration $F\xrightarrow{g} \pp^1$, 
coming from $X \to Z \to T$.
Note that $C_1$ is horizontal over $\pp^1$, as $C_1^2 > 0$.
Also, because $-(K_F + B_F - \epsilon C_2)$ is ample over $\pp^1$ 
and $K_F + B_F \equiv 0$, then $C_2$ is ample over $\pp^1$.
We conclude that $(X,B)$ has two horizontal components over $Z$, and
$(Z,B_Z)$ has complexity $\leq 1$, and therefore $(X,B)$ has complexity $\leq 1$.

As in both cases, $(X,B)$ has complexity $\leq 1$, the main result of~\cite{ELY25} implies that $(X,B)$ is cluster type.
\\

For the converse, notice that if (1) fails, 
then the pair is not cluster type by Theorem~\ref{thm:invariance-coreg}.
If (2) fails, then $(X,B)$ cannot be cluster type due to monodromy reasons.
Indeed, by~\cite[Theorem 3.4]{AdSFM24}, if the pair $(X,B)$ is of cluster type, then it is of cluster type over $(T,B_T)$ and hence there is a crepant birational map $\phi\colon (Y,B_Y)\dashrightarrow (X,B)$ where $(Y,B_Y)$ is a toric Calabi--Yau pair. 
Moreover, $\phi$ only extracts log canonical places so the strict transform to $Y$ of every component of $B$ is a component of $B_Y$. Since the torus invariant divisors on a toric morphism do not have monodromy, we conclude that the restriction of a component of $B$ to a general fiber of $X\rightarrow T$ must be irreducible. 

For the rest of the proof assume then that (1) and (2) hold, 
(3) does not hold, and $(X,B)$ is cluster type.
Then, the general fiber looks as in Figure~\ref{figure_GenFiberNot3}.

\begin{figure}[htb]
    \centering
    \hfill \begin{minipage}{.35\textwidth}
    \psscalebox{0.6 0.6} 
{
\begin{pspicture}(0,-3.1611795)(5.974321,1.1337886)
\psbezier[linecolor=black, linewidth=0.04](0.9357143,-2.4418018)(2.9738095,-2.2679923)(4.822619,-1.9918017)(5.0261903,1.1224839782714844)
\psbezier[linecolor=black, linewidth=0.04](1.5357143,-3.150135)(1.6880952,-0.23584935)(3.7916667,0.4153411)(5.963095,0.5486744544619568)
\rput[bl](2.2857144,-0.10025412){$C_1$}
\rput[bl](4.3428574,-1.9478731){$C_2$}
\rput[bl](0.0,-0.8812065){$(F,B_F)$}
\psdots[linecolor=black, dotsize=0.2](1.6380953,-2.3669207)
\psdots[linecolor=black, dotsize=0.2](4.958095,0.44107923)
\rput[bl](1.7353333,-2.749016){$p$}
\rput[bl](5.1113334,0.08298398){$q$}
\end{pspicture}
}
\end{minipage} 
\begin{minipage}{0.35\textwidth}
    with $C_1^2 \leq 0$ and $C_2^2 \leq 0$.
\end{minipage} \hfill
    \caption{General fiber assuming (1) and (2), but not (3).}
    \label{figure_GenFiberNot3}
\end{figure}

By Proposition~\ref{prop:toric-blow-up-criterion}, 
we can find a toric blow-up $\tilde{F} \to F$, 
a non-trivial effective divisor $m_1\tilde{C}_1 + m_2\tilde{C}_2$, 
supported on the strict transforms of $C_1$ and $C_2$, 
with $(m_1\tilde{C}_1 + m_2\tilde{C}_2)^2 = 0$, 
and a node outside $m_1\tilde{C}_1 + m_2\tilde{C}_2$.

First, notice that if $m_1 = 0$, 
then $m_2^2\tilde{C}_2 = 0$. But $\tilde{C}_2$ is the strict transform of $C_2$, 
and since $C_2^2 \leq 0$, $\tilde{C}_2^2 < 0$, which is a contradiction.

Now, assume that $m_1 > 0$ and $m_2 > 0$. 
Then, because there is a node outside $m_1\tilde{C}_1 + m_2\tilde{C}_2$, 
there must exist at least two curves in $B_{\tilde{F}} - \tilde{C}_1 - \tilde{C}_2$. 
We contract all the curves in $B_{\tilde{F}}$ except for $\tilde{C}_1$, $\tilde{C}_2$, 
and some other curve $E$ to obtain the diagram of Figure~\ref{figure_toricBlowup}

\begin{figure}[htb]
    \centering
    \psscalebox{0.6 0.6} 
{
\begin{pspicture}(0,-6.2191663)(20.822107,4.3774996)
\rput[bl](0.0,1.991748){$(F,B_F)$}
\psbezier[linecolor=black, linewidth=0.04](12.728772,3.3808331)(14.275438,4.6341662)(16.702105,2.6341665)(18.008772,4.207499796549477)
\psbezier[linecolor=black, linewidth=0.04](13.208772,-0.37916687)(14.7021055,-1.1525002)(17.555439,1.1141665)(18.408772,0.5008331298828052)
\psline[linecolor=black, linewidth=0.04](13.422105,3.9941664)(12.515439,2.2074997)
\psline[linecolor=black, linewidth=0.04](12.488772,2.8474998)(13.448771,1.2741665)
\psline[linecolor=black, linewidth=0.04](13.288772,2.0741665)(12.622106,0.39416647)
\psline[linecolor=black, linewidth=0.04](12.7021055,1.0341665)(13.662106,-0.7525002)
\psline[linecolor=black, linewidth=0.04](17.608772,4.1541667)(18.515438,2.9541664)
\psline[linecolor=black, linewidth=0.04](18.488771,3.2474997)(17.822105,1.9408331)
\psline[linecolor=black, linewidth=0.04](17.822105,2.4741664)(18.622105,1.0874997)
\psline[linecolor=black, linewidth=0.04](18.595438,1.7008331)(17.635439,0.10083313)
\rput[bl](15.368772,3.9674997){$\tilde{C}_1$}
\rput[bl](15.768772,-0.7525002){$\tilde{C}_2$}
\rput[bl](19.582106,2.2074997){$(\tilde{F},B_{\tilde{F}})$}
\psbezier[linecolor=black, linewidth=0.04](7.368772,-1.6058335)(10.648772,-2.0325003)(11.582106,-3.8725002)(12.008772,-5.019166870117192)
\psbezier[linecolor=black, linewidth=0.04](6.275439,-4.5925)(8.542106,-5.3925004)(10.782105,-5.3925004)(12.035439,-4.032500203450527)
\psline[linecolor=black, linewidth=0.04](6.542105,-5.2591667)(8.195438,-1.3125002)
\psdots[linecolor=black, dotsize=0.2](11.688772,-4.3525004)
\rput[bl](10.115438,-2.4058335){$\tilde{C}_1$}
\rput[bl](8.728772,-5.8191667){$\tilde{C}_2$}
\rput[bl](11.422105,-4.8858337){$q$}
\rput[bl](6.648772,-3.3658335){$E$}
\psline[linecolor=black, linewidth=0.04, arrowsize=0.013cm 5.0,arrowlength=2.0,arrowinset=0.6]{->}(11.395438,1.6208332)(7.662105,1.9674999)
\psline[linecolor=black, linewidth=0.04, arrowsize=0.013cm 5.0,arrowlength=2.0,arrowinset=0.6]{->}(6.0354385,-2.6458335)(4.7821054,0.3674998)
\psline[linecolor=black, linewidth=0.04, arrowsize=0.013cm 5.0,arrowlength=2.0,arrowinset=0.6]{->}(15.288772,-1.4191669)(12.355438,-3.259167)
\rput[bl](4.808772,-1.3925002){$\pi$}
\rput[bl](11.555439,-6.2191668){$(F',B_{F'})$}
\psbezier[linecolor=black, linewidth=0.04](2.2184346,2.2438087)(2.828961,1.61223)(5.6581106,1.0470477)(6.7739,2.857574020571066)
\psbezier[linecolor=black, linewidth=0.04](2.1909041,1.6591934)(3.180378,3.2381408)(5.48807,3.515064)(6.9617543,2.1045375833241047)
\psdots[linecolor=black, dotsize=0.2](2.4775438,2.03557)
\psdots[linecolor=black, dotsize=0.2](6.4985967,2.4566226)
\rput[bl](4.2670174,3.1513596){$C_1$}
\rput[bl](4.414386,1.172412){$C_2$}
\rput[bl](2.4564912,1.6355699){$p$}
\rput[bl](6.3933334,1.9724121){$q$}
\end{pspicture}
}
    \caption{Diagram of toric blow-ups when $m_1 > 0$ and $m_2 > 0$.}
    \label{figure_toricBlowup}
\end{figure}

As we contracted at least one curve intersecting 
$m_1\tilde{C}_1 + m_2\tilde{C}_2$, 
we have $(m_1\tilde{C}_1 + m_2\tilde{C}_2)^2 > 0$ in $F'$.

We claim that this cannot happen. 
Indeed, the map $F' \xrightarrow{\pi} F$ corresponds to a toric blow-up at $p$ 
of the ideal $(x^\alpha, y^\beta)$, 
where the local coordinates correspond to $C_1 = \{x = 0\}$ and $C_2 = \{y = 0\}$. 
A local computation shows that 
\[
    \tilde{C}_1\cdot E = \frac{1}{\beta}, 
    \qquad \tilde{C}_2\cdot E = \frac{1}{\alpha}, 
    \qquad \text{ and} 
    \qquad \tilde{C}_1\cdot \tilde{C}_2 = 1.
\]

Hence,
\[
    C_1^2   = \tilde{C}_1\cdot \pi^*(C_1) = \tilde{C}_1\cdot (\tilde{C}_1 + \alpha E) = \tilde{C}_1^2 + \frac{\alpha}{\beta},
\]
so $\tilde{C}_1^2 = C_1^2 - \frac{\alpha}{\beta} \leq -\frac{\alpha}{\beta}$.
Similarly, $\tilde{C}_2^2 \leq -\frac{\beta}{\alpha}$. Now we compute
\begin{align*}
    (m_1\tilde{C}_1 + m_2\tilde{C}_2)^2 &= m_1^2\tilde{C}_1^2 + 2m_1m_2 + m_2^2\tilde{C}_2^2\\
                                        &\leq -\frac{\alpha}{\beta}m_1^2 + 2m_1m_2 - \frac{\beta}{\alpha} m_2^2\\
                                        &= -\frac{1}{\alpha\beta}( \alpha m_1 - \beta m_2)^2 \leq 0.
\end{align*}
which is a contradiction.
\end{proof}

Finally we deduce Theorem~\ref{thm:ct-standard-model-Pic-1} from Theorem~\ref{thm:ct-standard-model-Pic-2}.

\begin{proof}[Proof of Theorem~\ref{thm:ct-standard-model-Pic-1}]
A similar computation as in the proof of 
Theorem~\ref{thm:ct-standard-model-Pic-2} yields that 
$B$ has one horizontal component. 

If (1), (2) and (3) hold, then on the general fiber $F$, $B_F$ 
is a nodal curve. Let $(X',B') \to (X,B)$ be the blow-up 
of the node on the general fiber. 
Then $(X',B')$ satisfies Conditions (1) and (2) of Theorem~\ref{thm:ct-standard-model-Pic-2}. 
As the volume of the fiber is $\geq 5$, the self-intersection of the strict transform of $B_F$ is $\geq 5 - 4 = 1$, 
so $(X',B')$ satisfies Condition (3) of 
Theorem~\ref{thm:ct-standard-model-Pic-2}. 
Therefore, the pair $(X',B')$ is of cluster type, so by Lemma~\ref{lem:ct-under-bir-comp2}, $(X,B)$ is also of cluster type.

Conversely, if $(X,B)$ is of cluster type, then the same argument as in the proof of Theorem~\ref{thm:ct-standard-model-Pic-2} holds to conclude that it satisfies Conditions (1) and (2). Let $(X',B') \to (X,B)$ be again the blow-up of the node in the general fiber. Then $(X',B')$ is of cluster type, and the strict transform of the nodal curve $B_F$ has positive self-intersection. Therefore, $\vol(F) = B_F^2 \geq 5$. 
\end{proof}

\section{Gorenstein del Pezzo surfaces}

In this section we prove Theorem \ref{thm:dP-Gor}.
In the proof, we follow the notation for singularities of Gorenstein del Pezzo surfaces (see, e.g.~\cite{MZ88}). First, we prove a theorem that allows us to detect whether any Calabi--Yau pair in such surfaces is of cluster type or not. 

\begin{theorem}\label{thm:complete-Gor-dP}
Let $X$ be a Gorenstein del Pezzo surface of Picard rank one. 
Asume that $(X,B)$ is a cluster type pair.
Then $X$ has only $A$-type singularities 
and precisely one of the following statements holds:
\begin{enumerate}
    \item $B$ has at least two components, 
    \item $B$ is an irreducible nodal curve, $B\subset X^{\rm sm}$, and $\vol(X)\geq 5$, 
    \item $B$ is an irreducible nodal curve with an $A_n$ singularity at the node
    and $\vol(X)\geq 3$, 
    \item $B$ is an irreducible nodal curve with an $A_n$ singularity at the node, 
    $n\geq 2$, and $\vol(X)=2$, or
    \item $B$ is an irreducible nodal curve with an $A_n$ singularity at the node, 
    $n\geq 4$, and $\vol(X)=1$. 
\end{enumerate}
\end{theorem}

\begin{proof}
If $B$ has at least two components, then we have $c(X,B)\leq 1$.
Therefore, $(X,B)$ is a cluster type by Theorem~\ref{thm:comp1-ct}.
Thus, we are in Case (1). 
From now on, we assume that $B$ is an irreducible nodal curve. 
If $B$ is contained in the smooth locus, then 
$(X,B)$ is a standard model over a point.
Therefore, we may apply Theorem~\ref{thm:ct-standard-model-Pic-1} 
to conclude that $(X,B)$ is of cluster type if and only $\vol(X)\geq 5$. 

Now, we assume that $B$ is an irreducible nodal curve that passes through a singular point of $X$.
By performing adjunction to the normalization of $B$,  we see that 
the singular point of $X$ must be in the node of $B$.
Thus, $B$ is an irreducible nodal curve with an $A_n$ singularity at the node. 
First, we assume that $\vol(X)\geq 3$. 
Let $\pi\colon Y\rightarrow X$ be the minimal resolution of the $A_n$ singularity at the node of $B$. 
Write $\pi^*(K_X+B)=K_Y+B_Y+E_1+\dots+E_n$. 
Note that $B_Y^2=B^2-2$ and so $B_Y^2\geq 1$. 
Let $\psi\colon Z\rightarrow Y$ be a sequence of blow-ups at the intersection of $B_Y$ and $E_1$ such that the strict transform
$B_Z$ of $B_Y$ in $Z$ satisfies that $B_Z^2=0$. 
By abuse of notation, we let $E_1,\dots,E_k$ be the exceptional divisors of $Z$ over $X$. 
Therefore, we have 
\[
\psi^*(\pi^*(K_X+B))=
K_Z+B_Z+E_1+\dots+E_k 
\]
where $k\geq n$ and $k\geq 2$. 
Therefore, we have an induced fibration 
$Z\rightarrow \pp^1$ that contracts $B_Z$. 
The Calabi--Yau pair $(Z,B_Z+E_1+\dots+E_k)$ has two horizontal components over $\pp^1$, namely, $E_1$ and $E_k$. 
Therefore, by Lemma~\ref{lem:conic-fib-ct}, we conclude that $(Z,B_Z+E_1+\dots+E_k)$
and so $(X,B)$ are of cluster type.
Here, we are in Case (3). 

Assume that $B$ is an irreducible nodal curve that passes through a singular point of $X$ and $\vol(X)=2$. 
Let $A_n$ be the singular point of $X$ at the node of $B$.
Let $\pi\colon Y\rightarrow X$ be the minimal resolution of such $A_n$ singularity.
Let $K_Y+B_Y+E_1+\dots+E_n=\pi^*(K_X+B)$
If $n\geq 2$, then $B_Y^2=0$ and
there is a node of $E_1+\dots+E_n$ that is not contained in $B_Y$. 
Therefore, we conclude that $(Y,B_Y+E_1+\dots+E_n)$ 
is of cluster type and so $(X,B)$
is of cluster type.  
If $n=1$, then $B_Y^2=0$ and $E_1^2=-2$. Therefore, the pair $(Y,B_Y+E_1)$ is a standard model over  a point.
By Theorem~\ref{thm:ct-standard-model-Pic-2}, we conclude that $(Y,B_Y+E_1)$ is not of cluster type. 
This proves Case (4). 

Finally, assume that $B$ is an irreducible nodal curve with an $A_n$ singularity at the node
and $\vol(X)=1$.
We show that $(X,B)$ is of cluster type if and only if $n\geq 4$. 
Let $\pi\colon Y\rightarrow X$ be the minimal resolution of the $A_n$-singularity at the node. 
Write $\pi^*(K_X+B)=K_Y+B_Y+E_1+\dots+E_n$. 
Here, we have $B_Y^2=-1$. 
Therefore, we can inductively 
blow-down $B_Y,E_1,\dots,E_{n-1}$
to smooth points.
Let $(Z,E_{Z,n})$ be the Calabi--Yau pair 
obtained by this procedure,
where $E_{Z,n}$ is the push-forward of $E_{n}$ to $Z$. 
Therefore, $Z$ is smooth along $E_{Z,n}$, $\rho(Z)=1$, $E_{Z,n}$ is a nodal curve, 
and $E_{Z,n}^2=-2+(n-1)+4 = n+1$.
Therefore, the pair $(Z,E_{Z,n})$ 
is an standard model. 
We then can apply Theorem~\ref{thm:ct-standard-model-Pic-1}
to conclude that $(Z,E_{Z,n})$
is of cluster type if and only if 
$n\geq 4$. 
Since the crepant birational map
$(X,B)\dashrightarrow (Z,E_{Z,n})$ 
only extracts and contracts
log canonical places, 
we conclude that $(X,B)$ is
of cluster type if and only if $n\geq 4$.
\end{proof}

In the following proof, we will see 
that all the cases of Theorem~\ref{thm:complete-Gor-dP} do actually happen.

\begin{proof}[Proof of Theorem~\ref{thm:dP-Gor}]
Let $X$ be a Gorenstein del Pezzo surface of Picard rank one. 
If $X$ is cluster type, then $X$ only has $A$-type singularities \cite[Lemma 5.1]{EFM24}. By \cite{MZ88}, we know that if $X$ is toric, then either $X\simeq \pp^2$
or their singularity types are $3A_2$, $A_3 + 2A_1$, $A_2 + A_1$ or $A_1$.
All these toric surfaces satisfy that $\vol(X)>1$.
Therefore, we are left with checking the theorem for the following classes of singularities: \\
\[
\begin{array}{|c| c | c|} 
\hline  
{\rm vol}(X)\geq 2 &  {\rm vol}(X)=1 \text{ and } |X^{\rm sing}|\leq 3 & {\rm vol}(X)=1 \text{ and }
|X^{\rm sing}|=4 \\
\hline 
   X(A_4) &  X(A_8)  &  X(2A_1+2A_3) \\
   X(A_7) &  X(A_1+A_7) & X(4A_2) \\
   X(A_1+A_5) &  X(2A_4) &  \\
   X(A_2+A_5) &  X(A_1+A_2+A_5)  &  \\
   X(A_1+2A_3) & & \\ 
    \hline 
\end{array}
\]
\hspace{0.2cm}
Recall that the isomorphism class of these surfaces are determined by their singularities.
We will proceed as follows.
First, we will argue that $X(4A_2)$
and $X(2A_1+2A_3)$ are not of cluster type. 
Secondly, we want to prove that all other surfaces in the table are of cluster type. 

First, we argue that $X(2A_1+2A_3)$ is not of cluster type. Throughout this paragraph we set $X:=X(2A_1+2A_3)$.  
Assume that $(X,B)$ is a Calabi--Yau pair
of index one and coregularity zero. 
First, we argue that $B$ has one component. Note that $B$ cannot have three components as in such case $(X,B)$ would be toric, leading to a contradiction. 
Assume that $B=B_{1}+B_{2}$ is the prime decomposition. As $(X,B)$ is Calabi--Yau and $X$ has only $A$-type singularities, the curves $B_{1}$ and $B_{2}$ intersect at two points $p$ and $q$. Note that 
\begin{equation}\label{ineq}
1 = (-K_{X})^2 = (B_{1}+B_{2})^2 
> 2B_{1}B_{2} = 2\left( 
\frac{1}{n+1}+\frac{1}{m+1}
\right) 
\end{equation} 
where $A_n$ and $A_m$ are the singularities of $X$ at the intersection points of $B_{1}$ and $B_{2}$. As $n,m\in \{1,3\}$, we get that the right hand side of~\eqref{ineq} is at least one, leading to a contradiction. 
Thus, we conclude that $B$ must be a nodal irreducible curve. As $(X,B)$ is Calabi--Yau, the only singularities of $X$ can happen at the node of $B$.
Hence, we conclude that $(X,B)$ is not of cluster type by Theorem~\ref{thm:complete-Gor-dP}.(2) and (5). We conclude that the surface $X(2A_1+2A_3)$ is not of cluster type. 

Now, we turn to argue that $X(4A_2)$ is not of cluster type. The argument is analogous to the previous case. Throughout this paragraph, we set
$X:=X(4A_2)$. Assume that $(X,B)$ is a Calabi--Yau pair of index one and coregularity zero. 
First, we argue that $B$ is irreducible.
Otherwise, we can write $B=B_1+B_2$ for its irreducible components, and so 
$1=(-K_X)^2=(B_1+B_2)^2 > 2B_1B_2 \geq 4/3$, leading to a contradiction. 
Thus, $B$ is an irreducible nodal curve and
$X$ can only have $A$-type singularities at the node of $B$. 
By Theorem~\ref{thm:complete-Gor-dP}.(2) and (5), we conclude that $X$ is not of cluster type.

We turn to prove that all the other surfaces in the previous table are of cluster type. 
Throughout the rest of the proof we consider the minimal resolutions of Gorenstein
del Pezzo of Picard rank 1 and of volume 1 listed in \cite[Figure 1 \& Figure 1']{MZ88}.
First, consider $X(A_7)$ and its minimal resolution $Y(A_7)$.
The smooth surface $Y(A_7)$ admits a contraction to $Y(A_1+A_5)$, the minimal resolution
of $X(A_1+A_5)$, as shown in Figure \ref{fig:contractionA7toA1A5},
\begin{figure}[hbt]
    \centering
    \psscalebox{0.7 0.7} 
{
\begin{pspicture}(0,-4.46)(19.15,3.42)
\psline[linecolor=black, linewidth=0.04](1.55,2.94)(5.55,2.94)
\psline[linecolor=black, linewidth=0.04](0.55,1.94)(0.55,-2.06)
\psline[linecolor=red, linewidth=0.04](1.55,-3.06)(5.55,-3.06)
\psline[linecolor=black, linewidth=0.04](6.55,-2.06)(6.55,1.94)
\psline[linecolor=black, linewidth=0.04](4.55,3.34)(6.95,0.94)
\psline[linecolor=black, linewidth=0.04](2.55,3.34)(0.15,0.94)
\psline[linecolor=black, linewidth=0.04](2.55,-3.46)(0.15,-1.06)
\psline[linecolor=black, linewidth=0.04](4.55,-3.46)(6.95,-1.06)
\psline[linecolor=black, linewidth=0.04, linestyle=dashed, dash=0.17638889cm 0.10583334cm](0.15,-0.06)(2.55,-0.06)
\psline[linecolor=black, linewidth=0.04, linestyle=dashed, dash=0.17638889cm 0.10583334cm](4.55,-0.06)(7.05,-0.06)
\psline[linecolor=black,linewidth=0.04, arrowsize=0.013cm 5.0,arrowlength=2.0,arrowinset=0.6]{->}(8.15,-0.06)(11.15,-0.06)
\rput[bl](3.3921053,-3.5573685){$C$}
\rput[bl](0.75,-2.61){$E_1$}
\rput[bl](0.0,-0.76){$E_2$}
\rput[bl](0.9,2.19){$E_3$}
\rput[bl](3.3,3.09){$E_4$}
\rput[bl](5.75,2.24){$E_5$}
\rput[bl](6.7,0.34){$E_6$}
\rput[bl](6.45,-2.66){$E_7$}
\rput[bl](1.4,0.04){$F_1$}
\rput[bl](5.35,0.09){$F_2$}
\psline[linecolor=black, linewidth=0.04](13.55,2.94)(17.55,2.94)
\psline[linecolor=black, linewidth=0.04, linestyle=dashed, dash=0.17638889cm 0.10583334cm](12.55,1.94)(12.55,-2.06)
\psline[linecolor=red, linewidth=0.04](13.55,-3.06)(17.55,-3.06)
\psline[linecolor=black, linewidth=0.04](18.55,-2.06)(18.55,1.94)
\psline[linecolor=black, linewidth=0.04](16.55,3.34)(18.95,0.94)
\psline[linecolor=black, linewidth=0.04](14.55,3.34)(12.15,0.94)
\psline[linecolor=black, linewidth=0.04](14.55,-3.46)(12.15,-1.06)
\psline[linecolor=black, linewidth=0.04](16.55,-3.46)(18.95,-1.06)
\psline[linecolor=black, linewidth=0.04, linestyle=dashed, dash=0.17638889cm 0.10583334cm](16.55,-0.06)(19.05,-0.06)
\rput[bl](15.392105,-3.5573685){$C$}
\rput[bl](12.75,-2.61){$E_1$}
\rput[bl](11.9,-0.56){$E_2$}
\rput[bl](12.9,2.19){$E_3$}
\rput[bl](15.3,3.09){$E_4$}
\rput[bl](17.75,2.24){$E_5$}
\rput[bl](18.7,0.34){$E_6$}
\rput[bl](18.45,-2.66){$E_7$}
\rput[bl](17.35,0.09){$F_2$}
\rput[bl](3.05,-4.46){$Y(A_7)$}
\rput[bl](14.65,-4.46){$Y(A_1 + A_5)$}
\end{pspicture}
}
    \caption{Contraction from $Y(A_7)$ to $Y(A_1 + A_5)$}
    \label{fig:contractionA7toA1A5}
\end{figure}
where the solid lines are $(-2)$-curves
and the dotted lines are $(-1)$-curves.
The curve $C$ in the picture has self-intersection zero.
Consider the $1$-complement of $Y$ 
given by $B_{Y(A_7)}:=E_1+\dots+E_7+C$. 
The push-forward $B_{X(A_7)}$ of $B_{Y(A_7)}$
to $X(A_7)$ is a nodal curve with an $A_7$-singularity at the node. 
By Theorem~\ref{thm:complete-Gor-dP}.(4), we conclude that the pair $(X(A_7),B_{X(A_7)})$ is a cluster type pair.  
Let $B_{X(A_1+A_5)}$ be the push-forward of $B_{Y(A_7)}$ to $X(A_1+A_5)$.
Then, the boundary $B_{X(A_1+A_5)}$ has two components, and so $(X(A_1+A_5),B_{X(A_1+A_5)})$
is of cluster type by Theorem~\ref{thm:complete-Gor-dP}.(1).

Consider the surface $X(A_8)$
and its minimal resolution $Y(A_8)$. 
In the surface $Y(A_8)$, we can contract a
$(-1)$-curve to obtain $Y(A_2+A_5)$, the minimal resolution of $X(A_2+A_5)$, 
as shown in Figure \ref{fig:contractionA8toA2A5} 
\begin{figure}[htb]
    \centering
    \psscalebox{0.7 0.7} 
{
\begin{pspicture}(0,-4.46)(18.426641,3.42)
\psline[linecolor=black, linewidth=0.04](1.026641,2.94)(6.026641,2.94)
\psline[linecolor=black, linewidth=0.04](0.026641006,1.34)(2.026641,-1.46)
\psline[linecolor=black, linewidth=0.04](0.026641006,-3.06)(5.026641,-3.06)
\psline[linecolor=black, linewidth=0.04](5.826641,-1.86)(5.826641,0.74)
\psline[linecolor=black, linewidth=0.04](4.026641,1.94)(6.026641,0.04)
\psline[linecolor=black, linewidth=0.04](2.026641,3.34)(0.026641006,0.44)
\psline[linecolor=black, linewidth=0.04](0.026641006,-3.56)(2.026641,-0.56)
\psline[linecolor=black, linewidth=0.04](4.026641,-3.46)(6.126641,-1.06)
\psline[linecolor=black,linewidth=0.04, arrowsize=0.013cm 5.0,arrowlength=2.0,arrowinset=0.6]{->}(7.626641,-0.06)(10.626641,-0.06)
\rput[bl](2.2687464,-3.5573685){$E_4$}
\rput[bl](0.326641,-2.31){$E_5$}
\rput[bl](0.776641,-0.96){$E_6$}
\rput[bl](0.376641,1.89){$E_7$}
\rput[bl](3.276641,3.09){$E_8$}
\rput[bl](5.126641,1.14){$E_1$}
\rput[bl](5.976641,-0.66){$E_2$}
\rput[bl](4.926641,-2.86){$E_3$}
\rput[bl](2.526641,-4.46){$Y(A_8)$}
\psline[linecolor=black, linewidth=0.04, linestyle=dashed, dash=0.17638889cm 0.10583334cm](5.626641,3.24)(4.026641,0.94)
\psline[linecolor=black, linewidth=0.04, linestyle=dashed, dash=0.17638889cm 0.10583334cm](0.126641,-0.66)(1.926641,1.54)
\rput[bl](1.526641,0.54){$F_1$}
\psline[linecolor=black, linewidth=0.04, linestyle=dashed, dash=0.17638889cm 0.10583334cm](5.526641,-2.36)(4.026641,-0.96)
\rput[bl](5.226641,2.14){$F_2$}
\rput[bl](3.926641,-1.76){$F_3$}
\psline[linecolor=black, linewidth=0.04](13.026641,2.94)(18.026642,2.94)
\psline[linecolor=black, linewidth=0.04, linestyle=dashed, dash=0.17638889cm 0.10583334cm](12.026641,1.34)(14.026641,-1.46)
\psline[linecolor=black, linewidth=0.04](12.026641,-3.06)(17.026642,-3.06)
\psline[linecolor=black, linewidth=0.04](17.826641,-1.86)(17.826641,0.74)
\psline[linecolor=black, linewidth=0.04](16.026642,1.94)(18.026642,0.04)
\psline[linecolor=black, linewidth=0.04](14.026641,3.34)(12.026641,0.44)
\psline[linecolor=black, linewidth=0.04](12.026641,-3.56)(14.026641,-0.56)
\psline[linecolor=black, linewidth=0.04](16.026642,-3.46)(18.12664,-1.06)
\rput[bl](14.268746,-3.5573685){$E_4$}
\rput[bl](12.326641,-2.31){$E_5$}
\rput[bl](12.476641,-0.56){$E_6$}
\rput[bl](12.376641,1.89){$E_7$}
\rput[bl](15.276641,3.09){$E_8$}
\rput[bl](17.12664,1.14){$E_1$}
\rput[bl](17.97664,-0.66){$E_2$}
\rput[bl](16.926641,-2.86){$E_3$}
\rput[bl](14.126641,-4.46){$Y(A_2 + A_5)$}
\psline[linecolor=black, linewidth=0.04, linestyle=dashed, dash=0.17638889cm 0.10583334cm](17.62664,3.24)(16.026642,0.94)
\psline[linecolor=black, linewidth=0.04, linestyle=dashed, dash=0.17638889cm 0.10583334cm](17.526642,-2.36)(16.026642,-0.96)
\rput[bl](17.22664,2.14){$F_2$}
\rput[bl](15.926641,-1.76){$F_3$}
\end{pspicture}
}
    \caption{Contraction from $Y(A_8)$ to $Y(A_2 + A_5)$}
    \label{fig:contractionA8toA2A5}
\end{figure}
We consider $B_{Y(A_8)}:=E_1+\dots+E_8+F_2$.
Let $B_{X(A_8)}$ be the push-forward of $B_{Y(A_8)}$ to $X(A_8)$.
The pair $(Y(A_8),B_{Y(A_8)})$ is Calabi--Yau of index one and coregularity zero. 
The push-forward $B_{X(A_2+A_5)}$ of $B_{Y(A_8)}$ to $X(A_2+A_5)$
has two components. 
Therefore, the pair $(X(A_2+A_5),B_{X(A_2+A_5)})$
and so $(X(A_8),B_{X(A_8)})$ are cluster type pairs
by Theorem~\ref{thm:complete-Gor-dP}.(1) and (5).

We show that $X(A_1+A_7)$ is a cluster type surface. Let $Y(A_1+A_7)$ be the minimal resolution of $X(A_1+A_7)$. 
The surface $Y(A+_1+A_7)$ admits a contraction of a $(-1)$-curve to $Y(A_1+2A_3)$, the minimal resolution of $X(A_1+2A_3)$, as shown in Figure \ref{fig:contractionA1A7toA12A3}.
\begin{figure}[htb]
    \centering
    \psscalebox{0.7 0.7} 
{
\begin{pspicture}(0,-4.46)(19.22,3.42)
\psline[linecolor=black, linewidth=0.04](1.55,2.94)(5.55,2.94)
\psline[linecolor=black, linewidth=0.04](0.55,1.94)(0.55,-2.06)
\psline[linecolor=black, linewidth=0.04](1.55,-3.06)(5.55,-3.06)
\psline[linecolor=black, linewidth=0.04, linestyle=dashed, dash=0.17638889cm 0.10583334cm](6.55,-2.06)(6.55,1.94)
\psline[linecolor=black, linewidth=0.04](4.55,3.34)(6.95,0.94)
\psline[linecolor=black, linewidth=0.04](2.55,3.34)(0.15,0.94)
\psline[linecolor=black, linewidth=0.04](2.55,-3.46)(0.15,-1.06)
\psline[linecolor=black, linewidth=0.04](4.55,-3.46)(6.95,-1.06)
\psline[linecolor=black, linewidth=0.04, linestyle=dashed, dash=0.17638889cm 0.10583334cm](0.15,-0.06)(2.55,-0.06)
\psline[linecolor=black,linewidth=0.04, arrowsize=0.013cm 5.0,arrowlength=2.0,arrowinset=0.6]{->}(8.15,-0.06)(11.15,-0.06)
\rput[bl](2.8921053,-3.5573685){$E_6$}
\rput[bl](0.75,-2.61){$E_5$}
\rput[bl](0.0,-0.76){$E_4$}
\rput[bl](0.9,2.19){$E_3$}
\rput[bl](2.8,3.09){$E_2$}
\rput[bl](5.75,2.24){$E_1$}
\rput[bl](6.8,-0.26){$F_4$}
\rput[bl](5.75,-2.66){$E_7$}
\rput[bl](1.4,0.14){$F_1$}
\rput[bl](2.65,-4.46){$Y(A_1 + A_7)$}
\psbezier[linecolor=black, linewidth=0.04, linestyle=dashed, dash=0.17638889cm 0.10583334cm](4.15,3.34)(3.75,2.94)(3.35,1.04)(4.25,0.54)
\psbezier[linecolor=black, linewidth=0.04](4.25,1.24)(3.75,1.04)(3.35,-0.76)(4.05,-1.46)
\psbezier[linecolor=black, linewidth=0.04, linestyle=dashed, dash=0.17638889cm 0.10583334cm](4.15,-0.86)(3.45,-1.56)(3.55,-3.06)(4.05,-3.56)
\rput[bl](3.85,1.84){$F_2$}
\rput[bl](3.85,-0.26){$E_8$}
\rput[bl](3.95,-2.26){$F_3$}
\psline[linecolor=black, linewidth=0.04](13.55,2.94)(17.55,2.94)
\psline[linecolor=black, linewidth=0.04, linestyle=dashed, dash=0.17638889cm 0.10583334cm](12.55,1.94)(12.55,-2.06)
\psline[linecolor=black, linewidth=0.04](13.55,-3.06)(17.55,-3.06)
\psline[linecolor=black, linewidth=0.04, linestyle=dashed, dash=0.17638889cm 0.10583334cm](18.55,-2.06)(18.55,1.94)
\psline[linecolor=black, linewidth=0.04](16.55,3.34)(18.95,0.94)
\psline[linecolor=black, linewidth=0.04](14.55,3.34)(12.15,0.94)
\psline[linecolor=black, linewidth=0.04](14.55,-3.46)(12.15,-1.06)
\psline[linecolor=black, linewidth=0.04](16.55,-3.46)(18.95,-1.06)
\rput[bl](14.892105,-3.5573685){$E_6$}
\rput[bl](12.75,-2.61){$E_5$}
\rput[bl](12.0,-0.26){$E_4$}
\rput[bl](12.9,2.19){$E_3$}
\rput[bl](14.8,3.09){$E_2$}
\rput[bl](17.75,2.24){$E_1$}
\rput[bl](18.8,-0.26){$F_4$}
\rput[bl](17.75,-2.66){$E_7$}
\psbezier[linecolor=black, linewidth=0.04, linestyle=dashed, dash=0.17638889cm 0.10583334cm](16.15,3.34)(15.75,2.94)(15.35,1.04)(16.25,0.54)
\psbezier[linecolor=black, linewidth=0.04](16.25,1.24)(15.75,1.04)(15.35,-0.76)(16.05,-1.46)
\psbezier[linecolor=black, linewidth=0.04, linestyle=dashed, dash=0.17638889cm 0.10583334cm](16.15,-0.86)(15.45,-1.56)(15.55,-3.06)(16.05,-3.56)
\rput[bl](15.85,1.84){$F_2$}
\rput[bl](15.85,-0.26){$E_8$}
\rput[bl](15.95,-2.26){$F_3$}
\rput[bl](14.55,-4.46){$Y(A_1 + 2A_3)$}
\end{pspicture}
}
    \caption{Contraction from $Y(A_1 + A_7)$ to $Y(A_1 + 2A_3)$}
    \label{fig:contractionA1A7toA12A3}
\end{figure}
We consider $B_{Y(A_1+A_7)}:=E_1+\dots+E_7+F_4$
a $1$-complement of coregularity zero.
The push-forward $B_{X(A_1+A_7)}$ of
$B_{Y(A_1+A_7)}$ to $X(A_1+A_7)$ is a 
nodal curve with an $A_7$ singularity at the node.
The push-forward $B_{X(A_1+2A_3)}$
of $B_{Y(A_1+A_7)}$ has to $X(A_1+2A_3)$
has two prime components.
By Theorem~\ref{thm:complete-Gor-dP}.(1) and (5), 
we conclude that both $X(A_1+A_7)$
and $X(A_1+2A_3)$ are of cluster type. 

Consider the surface $X(2A_4)$. 
Let $Y(2A_4)$ be its minimal resolution.
The smooth surface $Y(2A_4)$ admits a contraction of four $(-1)$-curves to the surface $Y(A_4)$, the minimal resolution of $X(A_4)$, as shown in Figure \ref{fig:contraction2A4toA4}, where $F_2$ has self-intersection zero.
\begin{figure}
    \centering
    \psscalebox{0.7 0.7} 
{
\begin{pspicture}(0,-4.26)(18.531866,3.22)
\psline[linecolor=black, linewidth=0.04](1.25,2.74)(6.25,2.74)
\psline[linecolor=black, linewidth=0.04](0.55,1.64)(0.55,-2.06)
\psline[linecolor=black, linewidth=0.04](4.25,-1.06)(6.25,0.54)
\psline[linecolor=black, linewidth=0.04](4.25,1.54)(6.25,-0.06)
\psline[linecolor=black, linewidth=0.04](2.25,3.14)(0.25,1.14)
\psline[linecolor=black, linewidth=0.04](2.25,-2.96)(0.25,-1.36)
\psline[linecolor=black, linewidth=0.04](4.25,-0.46)(6.25,-2.06)
\rput[bl](6.9921055,-1.0573684){$E_5$}
\rput[bl](0.85,-2.71){$E_1$}
\rput[bl](0.0,-0.96){$E_2$}
\rput[bl](0.7,2.09){$E_3$}
\rput[bl](3.5,2.89){$E_4$}
\rput[bl](5.05,0.94){$E_6$}
\rput[bl](5.3,-0.66){$E_7$}
\rput[bl](4.75,-1.56){$E_8$}
\rput[bl](2.85,-4.26){$Y(2A_4)$}
\psline[linecolor=black, linewidth=0.04, linestyle=dashed, dash=0.17638889cm 0.10583334cm](6.25,3.04)(4.25,0.94)
\psline[linecolor=black, linewidth=0.04, linestyle=dashed, dash=0.17638889cm 0.10583334cm](0.35,0.24)(2.35,1.94)
\rput[bl](2.15,1.34){$F_1$}
\rput[bl](5.55,1.84){$F_2$}
\psline[linecolor=black, linewidth=0.04, linestyle=dashed, dash=0.17638889cm 0.10583334cm](6.25,-1.46)(4.25,-2.96)
\psbezier[linecolor=black, linewidth=0.04](1.05,1.24)(1.65,-1.86)(5.75,1.04)(6.95,0.94)(8.15,0.84)(7.5393524,-2.660337)(6.75,-3.26)(5.9606476,-3.859663)(5.0363183,-2.8634536)(4.25,-2.26)
\rput[bl](5.25,-2.66){$F_3$}
\psline[linecolor=black,linewidth=0.04, arrowsize=0.013cm 5.0,arrowlength=2.0,arrowinset=0.6]{->}(8.25,-0.26)(11.25,-0.26)
\psline[linecolor=black, linewidth=0.04](12.95,2.74)(17.95,2.74)
\psline[linecolor=black, linewidth=0.04](12.25,1.64)(12.25,-2.06)
\psline[linecolor=black, linewidth=0.04](13.95,3.14)(11.95,1.14)
\psline[linecolor=black, linewidth=0.04](13.95,-2.96)(11.95,-1.36)
\rput[bl](12.55,-2.71){$E_1$}
\rput[bl](11.7,-0.96){$E_2$}
\rput[bl](12.4,2.09){$E_3$}
\rput[bl](15.2,2.89){$E_4$}
\rput[bl](14.85,-4.26){$Y(A_4)$}
\psline[linecolor=red, linewidth=0.04](17.35,3.04)(17.35,0.14)
\rput[bl](13.05,1.04){$F_1$}
\rput[bl](17.65,1.54){$F_2$}
\psbezier[linecolor=black, linewidth=0.04](14.75,0.74)(13.85,0.64)(13.85,-1.16)(14.75,-1.26)(15.65,-1.36)(16.25,-0.26)(16.25,-0.26)(16.25,-0.26)(16.85,0.94)(17.75,0.74)(18.65,0.54)(18.85,-0.96)(17.75,-1.26)(16.65,-1.56)(15.95,1.04)(14.75,0.74)
\psbezier[linecolor=black, linewidth=0.04, linestyle=dashed, dash=0.17638889cm 0.10583334cm](12.05,0.64)(12.85,1.24)(14.35,0.64)(14.65,0.14)
\rput[bl](17.45,-1.76){$E_5$}
\end{pspicture}
}
    \caption{Contraction from  $Y(2A_4)$ to $Y(A_4)$}
    \label{fig:contraction2A4toA4}
\end{figure}
We consider the boundary divisor
$B_{Y(2A_4)}:=E_5+E_6+E_7+E_8+F_3$. 
The pair $(Y(2A_4),B_{Y(2A_4)})$
is a Calabi--Yau pair of index one and
coregularity zero.
The push-forward $B_{X(2A_4)}$ of $B_{Y(2A_4)}$ 
to $X(2A_4)$ is a nodal curve with an $A_4$ singularity at the node.
Therefore, the pair $(X(2A_4),B_{X(2A_4)})$ is a cluster type pair by Theorem~\ref{thm:complete-Gor-dP}.(3).
Let $B_{X(A_4)}$ be the push-forward of $B_{Y(2A_4)}$ to $X(A_4)$. 
Then, $B_{X(A_4)}$ is a nodal curve contained in the smooth locus of $X(A_4)$ and so $(X(A_4),B_{X(A_4)})$ is of cluster type 
by Theorem~\ref{thm:complete-Gor-dP}.(2). 

Finally, we prove that $X(A_1+A_2+A_5)$ is a cluster type surface. 
Let $B$ be a nodal sextic in $\pp(1,2,3)$ that does not pass through the singular points. 
Then, the pair $(\pp(1,2,3),B)$
is a Calabi--Yau pair of index one and coregularity zero. This pair is of cluster type
by Theorem~\ref{thm:complete-Gor-dP}. 
Note that $B^2=6$.
We blow-up the node of $B$ 
and extract the $(-1)$-curve $E$. Then, we inductively blow-up
the intersection of the strict transform of $B$ and the previous $(-1)$-curve, until the self-intersection of the strict transform of $B$ equals $-2$. 
Afterwards, we blow-down all the $(-2)$-curves contained in the pull-back of $B$ to this blow-up. 
By doing so, we obtain a Calabi--Yau pair 
$(X(A_1+A_2+A_5),\Gamma)$ where $\Gamma$ is a nodal curve with an $A_5$-singularity at its node. 
Therefore, by Theorem~\ref{thm:complete-Gor-dP}.(5), we conclude that $X(A_1+A_2+A_5)$ is a cluster type surface.
\end{proof}

\section{Examples and questions} 
\label{sec:ex-quest}

In this section, we give some examples and pose some questions for further research. 
Example~\ref{ex:comp2-not-rat}, shows that pairs of complexity two may not be rational.
Example~\ref{ex:non-ct}
and Example~\ref{ex:ct} give some explicit applications of Theorem~\ref{thm:ct-standard-model-Pic-2} and
Theorem~\ref{thm:ct-standard-model-Pic-1}.

\begin{example}\label{ex:comp2-not-rat}
{\em
In this example, we show that there are Calabi--Yau pairs of index one, coregularity zero, and complexity two which are not rational.
Let $X_3\subset \pp^4$ be a smooth cubic $3$-fold.
Let $H$ be a general hyperplane in $\pp^3$ and $F$ be its restriction to $X_3$. Therefore, $F$ is a smooth cubic surface.
Therefore, the smooth surface $F$ admits a $1$-complement $B_F$ of coregularity zero.
Since the pair $(X_3,F)$ is log Fano, we can lift the complement $B_F$ to a $1$-complement $(X_3,F+B)$ of index one and coregularity zero. 
Note that $c(X_3,F+B)=2$.
Indeed, if $B$ has at least two prime components, then by Theorem~\ref{thm:comp1-ct} we would conclude that $X_3$ is a rational variety leading to a contradiction.
Hence, the Calabi--Yau pair $(X_3,F+B)$ has coregularity zero, index one, and complexity two.
}
\end{example}

\begin{example}\label{ex:non-ct}
{\em 
In this example, we show how to use Theorem~\ref{thm:ct-standard-model-Pic-1} or Theorem~\ref{thm:ct-standard-model-Pic-2} to detect whether a Calabi--Yau pair is not of cluster type. 
We consider the projective space $\pp^3$ with variables $[t:x:y:z]$.
Consider the Calabi--Yau pair
\[
(\pp^3,H+C_3) 
\]
where 
\[
H:=\{t=0\}
\quad 
\text{ and }
\quad 
C_3:=\{x(x^2+y^2+z^2)+2tyz=0\}.
\]
The pair $(\pp^3,H+C_3)$ is Calabi--Yau of index one and coregularity zero. 
The surface $C_3$ is smooth along $H$, therefore, the general element of the pencil 
\[
D_\lambda :=
\{ 
\lambda t^3 + (1-\lambda)(
x(x^2+y^2+z^2)+2tyz)=0
\}  
\]
is a smooth cubic surface
containing the reducible curve
$H\cap C_3$.
By blowing up the intersection of $D_0$ and $D_1$ we obtain a fibration $\pi\colon X\rightarrow \pp^1$. 
We can blow down the reducible fibers of $\pi$ to obtain a fibration $\psi\colon X'\rightarrow \pp^1$ of relative Picard rank two. 
Let $(X',H'+C_3')$ be the pull-back of $(\pp^3,H+C_3)$ to $X'$.
Then, the reduced fiber over $\{0\}$ of $\psi$ equals $H'$
while the reduced fiber over $\{\infty\}$ of $\psi$ equals $C_3'$.
Therefore, the crepant fibration
\[
\psi\colon (X',H'+C_3') \rightarrow (\pp^1,\{0\}+\{\infty\}) 
\]
is a standard model
of relative Picard rank two
over a toric base.
The general fiber of $\psi$ is isomorphic to $(D_\lambda,C)$
where $D_\lambda$ is a smooth cubic surface and $C$ is a curve with two components $C_1+C_2$ which are smooth rational curves.
We conclude that $C_1^2+C_2^2=-1$.
Since a cubic surface contains no smooth rational curves with $C^2\leq -2$, we conclude that
$C_1^2=0$ and $C_2^2=-1$, 
up to possibly swapping the curves.
By Lemma~\ref{lem:ct-under-bir-comp2} and Theorem~\ref{thm:ct-standard-model-Pic-2}, we conclude that $(\pp^3,H+C_3)$ is not of cluster type. Equivalently, we can contract $C_2$ over $\pp^1$ to obtain a new crepant fibration
\[
\psi_0 \colon (X_0,H_0+C_{3,0})\rightarrow (\pp^1,\{0\}+\{\infty\})
\]
which is a standard model of relative Picard rank one. 
The volume of the general fiber of $\psi_0$ is $4$.
Then, by Lemma~\ref{lem:ct-under-bir-comp2} and Theorem~\ref{thm:ct-standard-model-Pic-1}, we conclude that $(\pp^3,H+C_3)$ is not of cluster type. 
}    
\end{example}

\begin{example}\label{ex:ct}
{\em 
In this example, we show how to use Theorem~\ref{thm:ct-standard-model-Pic-2} or Theorem~\ref{thm:ct-standard-model-Pic-1} to detect whether a Calabi--Yau pair is of cluster type.
We consider the projective space $\pp^3$ with variables $[t:x:y:z]$.
Consider the Calabi--Yau pair
\[
(\pp^3,H+C_3), 
\]
where
\[
H:=\{t=0\} 
\quad \text{ and }
\quad 
C_3:=\{y^2z-x^2(x+z)+zt^2=0\}.
\]
The pair $(\pp^3,H+C_3)$ is Calabi--Yau of index one and coregularity zero. Indeed, the intersection $H\cap C_3$ is a nodal cubic in $\pp^2$.
The general element of the pencil
\[
D_\lambda:=
\{ 
\lambda t^3 + (1-\lambda) 
(y^2z -x^2(x+z) +zt^2)=0 
\} 
\]
is a degenerate cubic surface with a single $A_1$ singularity at $[0:0:0:1]$.
By blowing up the intersection of $D_0$ and $D_1$ we obtain a fibration $X\rightarrow \pp^1$.
We can blow down the reducible fibers of $\pi$ to obtain a fibration $\psi\colon X'\rightarrow \pp^1$ of relative Picard rank one.
Let $(X',H'+C'_3)$ be the pull-back of $(\pp^3,H+C_3)$ to $X'$.
Then, the reduced fiber over $\{0\}$ of $\psi$ equals $H'$
while the reduced fiber over $\{\infty\}$ of $\psi$ equals $C'_3$.
Therefore, the crepant fibration 
\[
\psi\colon (X',H'+C'_3)\rightarrow (\pp^1,\{0\}+\{\infty\})
\]
is a standard model 
of relative Picard rank one
over a toric base.
The general fiber of $\psi$ is isomorphic to $(D_\lambda,C_\lambda)$ where $C_\lambda$ is a nodal curve
with a node at the $A_1$ singularity of $D_\lambda$.
The general fiber of $\psi$ has volume $3$ so $C_\lambda^2=3$.
However, we cannot apply Theorem~\ref{thm:ct-standard-model-Pic-1} as $C_\lambda$ is not contained in the smooth locus of $D_\lambda$.
Let $X''\rightarrow \pp^1$ be the family obtained by blowing up the one-dimensional log canonical center
of $(X',H'+C'_3)$ over $\pp^1$.
By doing so, we resolve the general fiber of $\psi$.
We obtain a new crepant fibration
\begin{equation}\label{eq:3rd-cf}
\pi\colon (X'',H''+E''+C''_3)\rightarrow (\pp^1,\{0\}+\{\infty\}) 
\end{equation} 
whose general fiber is the minimal resolution $(F,B_F)$ of $(D_\lambda,C_\lambda)$.
The divisor $B_F$ has two components;
the strict transform of $C_\lambda$ and $E|_F$.
In addition, crepant fibration~\eqref{eq:3rd-cf} 
is a standard model of relative Picard rank two over a toric base.
As $F$ is smooth, the divisor $B_F$ is contained in the smooth locus. 
Thus, we may apply Theorem~\ref{thm:ct-standard-model-Pic-2}.
Note that the strict transform of $C_\lambda$ in $F$ has self-intersection $1$.
Thus, we conclude that the pair $(X'',H''+E''+C''_3)$
is cluster type. 
By Lemma~\ref{lem:ct-under-bir-comp2}, we conclude that $(\pp^3,H+C_3)$ is of cluster type.
}
\end{example}

In the following examples, we show that the condition on the base 
of Theorem~\ref{thm:invariance-coreg} is necessary.

\begin{example}
{\em 
In this example, we show that the conclusion of Theorem~\ref{thm:invariance-coreg} is not valid if the base of the crepant fibration is not a toric Calabi--Yau pair.

Let $X:=\pp^1\times \pp^1$.
Let $p\in X$
and consider $F\in |\mathcal{O}(0,1)|$ passing through $p$.
Let $D_1$ and $D_2$ be
two general elements of $|\mathcal{O}(2,1)|$ that pass through $p$ as well. 
Then, the pair
\begin{equation}\label{ex:not-inv-coreg}
\left(X,F + \frac{1}{2}D_1 + \frac{1}{2}D_2\right)     
\end{equation}
is a Calabi--Yau pair.
Consider $\pi_2\colon X\rightarrow \pp^1$ to be the projection onto the second component
and let $q$ be the image of $p$.
Then, we have
a crepant fibration 
\begin{equation}\label{eq:fibration-crepant-coreg-inv} 
\pi_1\colon 
\left(X,
F + \frac{1}{2}D_1+\frac{1}{2}D_2 
\right) 
\rightarrow 
\left( 
\pp^1,
\frac{1}{2}(s_1+s_2)+q 
\right) 
\end{equation} 
where $s_1,s_2\in \pp^1$ are points different from $q$. 
The pair $(\pp^1,(\frac{1}{2}(s_1+s_2)+q)$ is not toric although the ambient space is toric.
The coregularity of the Calabi--Yau pair~\eqref{ex:not-inv-coreg} is zero as $p$ is contained in $F$ and the sum of the multiplicities of the curves passing through $p$ is $2$. 
On the other hand, the general fiber of the crepant fibration~\eqref{eq:fibration-crepant-coreg-inv} is isomorphic to $(\pp^1,\frac{1}{2}(\lambda_1+\dots+\lambda_4))$ so its coregularity is one. 
This shows that the toricity of the Calabi--Yau pair in the base of the statement of Theorem~\ref{thm:invariance-coreg} is indeed necessary. 
}
\end{example}

Now, we turn to give some examples of surface pairs
$(X,B)$ and we analyze whether they are cluster type or not. 

\begin{example}\label{A-ct-nct}
{\em 
We give examples of two Calabi--Yau surfaces
$(X_1,B_1)$ and $(X_2,B_2)$, 
such that both $X_1$ and $X_2$ only have
$A$-type singularities,
$X_1\setminus B_1$ and $X_2\setminus B_2$
have the same singularities, 
and $(X_1,B_1)$ is of cluster type 
while $(X_2,B_2)$ is not of cluster type. 

Let $C_3\subset \pp^2$ be a nodal cubic curve. 
Then, the pair $(\pp^2,C_3)$ is of cluster type.  
We take two points $p,q\in C_3$ that are not in the node of $C_3$. 
We blow-up $p$ and then we blow-up the intersection of the exceptional divisor with the strict transform of $C_3$. 
By doing so, we introduce a $(-2)$-curve that we may blow-down to get a $A_1$-singularity.
Analogously, we can blow-up $q$ and then blow-up the intersection of the last exceptional divisor with the strict transform of $C_3$, five times. By doing so, we introduce five $(-2)$-curves that we may blow-down to get an $A_5$ singularity. Let $X_1$ be such a surface 
and let $(X_1,B_1)$ be the pull-back of $(\pp^2,C_3)$ to $X_1$.
Then, we have two singularities in $X_1\setminus B_1$, namely an $A_1$ singularity and an $A_5$-singularity. As $(\pp^2,C_3)$ is cluster type, we conclude that $(X_1,B_1)$ 
is of cluster type. 

On the other hand, consider $X_2:=X(A_1+A_2+A_5)$ and let $Y(A_1+A_2+A_5)$ be its minimal resolution as in Figure \ref{fig:minResA1A2A5}, 
\begin{figure}[hbt]
    \centering
    \psscalebox{0.7 0.7} 
{
\begin{pspicture}(0,-4.935)(7.1266413,3.395)
\psline[linecolor=black, linewidth=0.04](1.5,2.4277096)(5.5,2.4277096)
\psline[linecolor=black, linewidth=0.04](0.5,1.4277097)(0.5,-2.5722904)
\psline[linecolor=black, linewidth=0.04, linestyle=dashed, dash=0.17638889cm 0.10583334cm](6.5,-2.5722904)(6.5,1.4277097)
\psline[linecolor=black, linewidth=0.04](0.1,0.8277097)(2.5,2.8277097)
\psline[linecolor=black, linewidth=0.04](4.5,2.8277097)(6.9,0.8277097)
\psline[linecolor=black, linewidth=0.04](0.1,-1.9722904)(2.5,-3.9722903)
\psline[linecolor=black, linewidth=0.04](4.5,-3.9722903)(6.8,-2.0722904)
\psline[linecolor=black, linewidth=0.04, linestyle=dashed, dash=0.17638889cm 0.10583334cm](0.1,-0.97229034)(2.1,0.8277097)(2.1,0.8277097)
\psline[linecolor=black, linewidth=0.04](3.5,1.0277096)(3.5,-2.5722904)
\psbezier[linecolor=black, linewidth=0.04](0.8,0.12770966)(1.8,0.52770966)(3.8,-0.8722904)(4.9,-1.5722903442382812)(6.0,-2.2722902)(6.3,-1.8722904)(7.1,-0.6722903)
\psbezier[linecolor=black, linewidth=0.04, linestyle=dashed, dash=0.17638889cm 0.10583334cm](3.1,2.8277097)(2.5,2.0277097)(3.2,0.42770967)(3.9,0.12770965576171875)
\psbezier[linecolor=black, linewidth=0.04, linestyle=dashed, dash=0.17638889cm 0.10583334cm](2.9,-2.8722904)(3.3,-2.1722903)(4.7,-1.4722904)(5.4,-1.4722903442382813)
\rput[bl](1.0,-3.4722903){$E_1$}
\rput[bl](0.0,0.027709655){$E_2$}
\rput[bl](0.8,1.8277097){$E_3$}
\rput[bl](3.9,2.5277097){$E_4$}
\rput[bl](5.7,1.8277097){$E_5$}
\rput[bl](5.6,-3.4722903){$E_6$}
\rput[bl](5.2,-2.2722902){$E_7$}
\rput[bl](2.9,-1.6722903){$E_8$}
\rput[bl](0.9,-0.6722903){$F_1$}
\rput[bl](3.0,1.6277096){$F_2$}
\rput[bl](4.0,-2.3722904){$F_3$}
\rput[bl](6.7,0.12770966){$F_4$}
\rput[bl](2.1,-4.935){$Y(A_1 + A_2 + A_5)$}
\end{pspicture}
}
    \caption{Minimal resolution $Y(A_1 + A_2 + A_5)$ of $X(A_1 + A_2 + A_5)$}
    \label{fig:minResA1A2A5}
\end{figure}
where the solid lines represent $(-2)$-curves
and the dotted lines represent $(-1)$-curves. 

Consider the boundary divisor 
$B_{Y(A_1+A_2+A_5)}:=E_7+E_8+F_3$. 
Let $X_2$ be the surface obtained by contracting $F_3$ and then the image of $E_8$. 
Then, the surface $X_2$ is a Gorenstein del Pezzo surface of rank one with two singularities $A_1$ and $A_5$.
Let $B_2$ be the image of $E_7$ in $X_2$. 
Then, the pair $(X_2,B_2)$ is a Calabi--Yau surface and $B_2$ is a nodal curve.
Furthermore, $B_2$ is contained in the smooth locus of $X_2$. 
By Theorem~\ref{thm:complete-Gor-dP}.(2), 
we conclude that $(X_2,B_2)$ is not of cluster type.

We conclude that $(X_1,B_1)$ is of cluster type 
while $(X_2,B_2)$ is not of cluster type. 
However, they both have the same singularities along the complement of the boundary divisor. 
Note that $\vol(X_1)=1$ and $\vol(X_2)=3$
while $\rho(X_1)=3$ and $\rho(X_2)=1$. 
}
\end{example}

We finish the article by posing two questions. Calabi--Yau pairs of index one and complexity one are cluster type.
In this article, we explain that understanding which pairs of complexity two are cluster type
is a decidable problem. 
However, we did not attempt to understand the log rationality of pairs of complexity two.
This problem is probably more subtle and related to the rationality of conic fibrations. 
Nevertheless, we expect a positive answer to the following question. 

\begin{question}
Let $X$ be a rational variety of dimension $n$.
Let $(X,B)$ be a Calabi--Yau pair of index one, coregularity zero, and complexity two.
Is the pair $(X,B)$ log rational? i.e., is $(X,B)$ crepant birational to $(\pp^n,\Sigma^n)$? 
\end{question} 

The previous question is motivated by~\cite[Conjecture 1.1]{EFM24} and the work relating complexity and dual complexes due to the first author and Mauri~\cite{MM24}. 
Although Calabi--Yau pairs of complexity two may be rather complicated. We expect that pairs of complexity strictly less than two are approachable and may be classified in the future.

\begin{question}
Is it possible to give a characterization of Calabi--Yau pairs $(X,B)$, with arbitrary index, and $c(X,B)<2$? 
\end{question}

So far, only the case with $c(X,B)=0$ is fully understood by the work of Brown, McKernan, Svaldi, and Zong~\cite{BFMS20} and the work of Enwright and Figueroa~\cite{EF24}.
In the former, they prove that $X$ is a toric variety 
and in the latter, they prove that $B$ is a convex weighted combination of different toric boundaries of $X$ (with respect to different maximal tori of ${\rm Aut}(X)$). 
In the case of complexity strictly less than one a similar statement is expected, while in the case of complexity strictly less than two some new geometric behavior may occur.

\bibliographystyle{habbvr}
\bibliography{bib}
\end{document}